\documentclass[12pt]{amsart}
\usepackage{amsmath,amssymb,amsthm,amsfonts,enumerate,array,color,lscape,fancyhdr,layout,pst-all}
\usepackage[all]{xy}
\usepackage[hyperindex,bookmarksnumbered,plainpages]{hyperref}

\newtheorem{Th}{Theorem}[section]
\newtheorem{Pro}[Th]{Proposition}
\newtheorem{Lem}[Th]{Lemma}
\newtheorem{Co}[Th]{Corollary}
\newtheorem{Def}[Th]{Definition}
\theoremstyle{remark}
\newtheorem{Rem}[Th]{Remark}
\newtheorem{Ex}[Th]{Example}

\begin{document}
\title[Simple modules  in the minimal nilpotent orbit]{Simple modules for Affine vertex algebras in the minimal nilpotent orbit}
\author[V. Futorny]{Vyacheslav Futorny}
\address{Instituto de Matem\'atica e Estat\'istica, Universidade de S\~ao Paulo,  S\~ao Paulo, Brazil, and  International Center for Mathematics, SUSTech, Shenzhen, China}\email{vfutorny@gmail.com}
\author[O. A. Hern\'andez Morales]{Oscar Armando Hern\'andez Morales}
\address{Instituto de Matem\'atica e Estat\'istica, Universidade de S\~ao Paulo,  S\~ao Paulo, Brazil} \email{oscarhm@ime.usp.br}
\author[L. E. Ramirez]{Luis Enrique Ramirez}
\address{Universidade Federal do ABC, Santo Andr\'e, Brazil} \email{luis.enrique@ufabc.edu.br}

\begin{abstract}
 We explicitly construct,  in terms of Gelfand--Tsetlin tableaux, a new family of simple positive energy representations for the simple  affine vertex algebra $V_k(\mathfrak{sl}_{n+1})$  in the minimal  nilpotent orbit of $\mathfrak{sl}_{n+1}$.
These representations are  quotients of induced modules over the affine Kac-Moody algebra $\widehat{\mathfrak{sl}}_{n+1} $ and include in particular all admissible simple highest weight modules  
 and all simple modules induced from $\mathfrak{sl}_2$. Any such simple module in the minimal  nilpotent orbit  has bounded weight multiplicities. 
\end{abstract}

\subjclass[2010]{Primary 17B10, 17B69}
\keywords {Affine vertex algebra, localization functors, Gelfand--Tsetlin modules, tableaux realization.}

\maketitle

 \section{Introduction}
 Relaxed highest weight modules for affine Kac-Moody algebras attract a considerable interest due to their connection 
 to the representation theory of conformal vertex algebras and conformal field theories, cf.   
 \cite{FST1998}, \cite{AM95}, \cite{AM07}, \cite{AP08}, \cite{AM09}, \cite{Adamovic2016}, \cite{Auger-Creutzig-Ridout2018} and references therein. In particular, the importance of \emph{relaxed highest weight} admissible modules for conformal field theory was shown in \cite{R}, \cite{CR1}, \cite{CR2}, see also \cite{RW}, \cite{FKR}.
 Using the Zhu's functor the study of positive energy representations of simple affine vertex algebras reduces to the representations of the underlined finite-dimensional Lie algebra. This allows to construct new families of simple  
 representations of admissible vertex algebras. This has been exploited in \cite{AFR17}, where new families of simple 
 modules were constructed for the affine vertex algebra of $\mathfrak{sl}_n$ and a  complete classification of all simple relaxed highest weight representations  with finite-dimensional   weight spaces in the case $n=3$ was obtained. Further on, this approach was developed in \cite{KR19a} and \cite{KR19b} using Fernando-Mathieu classification of simple  weight representations with finite weight multiplicities of finite-dimensional simple Lie algebras. 
These papers provided an algorithm for classifying all admissible relaxed highest weight modules starting from admissible highest weight modules.  The latter modules were classified in \cite{KacWak89} and \cite{Ara16}.
However, explicit
construction of  admissible relaxed highest weight representations beyond their classification is rather difficult. Moreover, the algorithm is limited to the modules with finite weight multiplicities.  We refer to the recent paper \cite{FK20}, where the localization 
and the Wakimoto functors were used to construct \emph{relaxed Wakimoto modules} for the universal and simple affine vertex algebras with infinite weight multiplicities.

Current paper aims to give an explicit
construction of admissible representations  extending \cite{AFR17}, where the Gelfand--Tsetlin tableau realization was used 
to describe  admissible families in type $A$.  Such realization provides a basis and explicit formulas for the action of the Lie algebra.  In this paper we describe several new families of simple positive energy weight representations of the admissible affine vertex algebra $V_k(\mathfrak{sl}_{n})$. These families include all  simple highest weight modules in the minimal nilpotent orbit described in \cite{Ara15}.   They also include simple modules in the minimal nilpotent orbit induced from an $\mathfrak{sl}_2$-subalgebra. All constructed modules 
  have bounded weight multiplicities.
Our approach is based on the theory of \emph{relation} Gelfand--Tsetlin modules developed in \cite{FRZ19}, where
  explicit tableau basis was constructed for different classes of simple Gelfand--Tsetlin modules for $\mathfrak{gl}_n$.

Our first main result is Theorem \ref{hw module}, which provides a tableau realization for simple highest weight modules, namely we give necessary and sufficient conditions for such modules to be relation modules with respect to the standard Gelfand--Tsetlin subalgebra.
In Section \ref{Section: Induced relation modules} we apply twisted localization to relation Gelfand--Tsetlin modules. 
We establish the following result (cf. Theorem \ref{twistedE21}).

\
 
  \begin{Th}\label{thm-main-1} Let $\mathfrak{g}=\mathfrak{sl}_{n+1}$, $\Gamma$ the standard Gelfand--Tsetlin subalgebra (corresponding to the chain of embeddings starting from the upper left corner), $f=E_{21}$,
  $M$  a $\Gamma$-relation module with an injective action of $f$. Then the twisted localization $D_f^{x} M$ of $M$ is a $\Gamma$-relation Gelfand--Tsetlin module for any $x\in \mathbb C$.
\end{Th}  

\

Using the theorem above we obtain our second main result -  Theorem \ref{sl2 Induced}, where we classify all simple $\Gamma$-relation Gelfand--Tsetlin modules induced from $\mathfrak{sl}_2$-subalgebra generated by  $E_{12}$ and $E_{21}$. Also, in Theorem \ref{Lem: Induced by sl(m)} we obtain an explicit construction of a family of simple $\mathfrak{sl}_{n+1}$-modules, which  are parabolically induced from cuspidal $\mathfrak{a}$-modules, where $\mathfrak{a}\simeq \mathfrak{sl}_m$ is part of the $\Gamma$-flag,  $m=2, \ldots, n$. 
Further, localization of simple highest weight relation modules is considered in Proposition \ref{localiz-F}
for a multiplicative set generated by any number of commuting generators of the form $E_{k1}$.

Finally, in 
 Section  \ref{Section: Minimal orbit} we establish our third main result - explicit description of simple admissible highest weight $\mathfrak{sl}_{n+1}$-modules and modules induced from 
$\mathfrak{sl}_2$ subalgebra in  the unique minimal non-trivial nilpotent orbit 
 $\mathbb{O}_{min}$  of $\mathfrak{sl}_{n+1}$. Let
$$
k+n=\frac{p}{q}-1,\quad  p> n,\ q\geq 1 \mbox{ and }\ (p,q)=1,
$$
$a\in \{1,2,\dots ,q-1\}$ and 
$
\lambda_a=\lambda-\frac{ap}{q}\varpi_1$, where 
$\lambda=\left(\lambda_1, \ldots, \lambda_n\right)
$ is a weight of $\mathfrak{sl}_{n+1}$ with
 $\lambda_i\in\mathbb{Z}_{\geq 0}$, for all $ i=1,\dots, n$, $\lambda_{1}+\ldots+\lambda_n< p-n$. Then we have 
 (cf. Theorem \ref{Pro: HWM} and Theorem \ref{inducedsl2minimal}):

\
 
  \begin{Th}\label{thm-main-3} 
  \begin{itemize}
 \item[(i)] If $\Gamma$  is the standard Gelfand--Tsetlin subalgebra then 
  the simple highest weight module $L(\lambda_a)$ is a bounded $\Gamma$-relation Gelfand--Tsetlin $\mathfrak{sl}_{n+1}$-module. Moreover, all simple admissible highest weight  modules in the minimal nilpotent orbit are bounded $\Gamma$-relation Gelfand--Tsetlin modules.
  \item[(ii)] Let $\pi$ be a basis of the root system of $\mathfrak g$, $\mathfrak b(\pi)$  the corresponding Borel subalgebra  of $\mathfrak g$,
  $\beta$  a positive (with respect to $\pi$) root of $\mathfrak g$, 
  $\rho_{\pi}$ the half-sum of positive (with respect to $\pi$) roots.
 Let $L_{\mathfrak b(\pi)}(\lambda)$ be an admissible simple $\mathfrak b(\pi)$-highest weight $\mathfrak g$-module  in the minimal  orbit,  such that  $\left<\lambda+\rho_{\pi},\beta^{\vee} \right> \notin\mathbb{Z}$, and $f=f_{\beta}$. Denote by $A_{\pi, \beta}$ the set of all $x\in \mathbb{C}\setminus \mathbb{Z}$ for which   $x+\left<\lambda+\rho_{\pi},\beta^{\vee} \right> \notin\mathbb{Z}$.
 Then the twisted localization modules $D^x_{f}L_{\mathfrak b(\pi)}(\lambda)$, $x\in A_{\pi, \beta}$, exhaust all  simple $\mathfrak{sl}_2$-induced admissible modules in the minimal orbit. Moreover,
    there exists a flag ${\mathcal F}$ such that $D^x_{f}L_{\mathfrak b(\pi)}(\lambda)$ is  $\Gamma_{\mathcal F}$-relation Gelfand--Tsetlin $\mathfrak g$-module.
  \end{itemize}
 \end{Th}

\

\noindent{\bf Acknowledgements.}  V.F. and L.E.R. gratefully acknowledge the
hospitality and excellent working conditions of the International Center for Mathematics of SUSTech,  where part of this work was completed. V.F. is supported in part by the CNPq (304467/2017-0) and by the Fapesp grant (2018/23690-6); O.H. is supported by the Coordena\c{c}\~ao de Aperfei\c{c}oamento de Pessoal de N\'ivel Superior - Brasil (CAPES) - Finance Code 001; L.E.R. is supported  by the Fapesp grant (2018/17955-7)

\section{Preliminaries}\label{section:Preliminaries}

\subsection{Weight modules}

Let $\mathfrak g$ be a simple complex finite dimensional Lie algebra, $\mathfrak h$ a fixed Cartan subalgebra,
$U(\mathfrak g)$ the universal enveloping algebra of $\mathfrak g$, and  $W$ the Weyl group of $\mathfrak g$. 
By $\Delta$ we denote the root system of $\mathfrak g$ and by
  $\Delta_+$ the set of positive roots with respect to a fixed basis $\pi$ of $\Delta$.
Put $Q=\sum\limits_{\alpha\in \Delta}\mathbb Z \alpha$ for the root lattice and 
$Q^\vee=\sum\limits_{\alpha\in \Delta}\mathbb Z \alpha^{\vee}$  for the coroot lattice,
where $\alpha^{\vee}=2\alpha/(\alpha,\alpha)$. Let $\rho_{\pi}$ to be the half sum of positive roots. 

By $\mathfrak{b}(\pi)$ we denote the standard Borel subalgebra of $\mathfrak{g}$ corresponding to the set $\pi$. In addition to the standard Borel subalgebra of $\mathfrak{g}$ we  also consider the standard parabolic subalgebras of $\mathfrak{g}$. For a subset $\Sigma$ of $\pi$  denote by $\Delta_\Sigma$ the root subsystem in $\mathfrak{h}^*$ generated by $\Sigma$. Then the standard parabolic subalgebra $\mathfrak{p}_\Sigma(\pi)$ of $\mathfrak{g}$ associated to $\pi$ and $\Sigma$ is defined as $\mathfrak{p}_\Sigma(\pi) = \mathfrak{l}_\Sigma \oplus \mathfrak{u}^{+}_\Sigma$ with nilradical $\mathfrak{u}^{+}_\Sigma := \bigoplus_{\alpha \in \Delta_+ \setminus \Delta_\Sigma} \mathfrak{g}_\alpha$, opposite nilradical $\mathfrak{u}^{-}_\Sigma := \bigoplus_{\alpha \in \Delta_+ \setminus \Delta_\Sigma} \mathfrak{g}_{-\alpha}$,  and  Levi subalgebra $\mathfrak{l}_\Sigma$ defined by
\begin{align*}
  \mathfrak{l}_\Sigma := \mathfrak{h} \oplus \bigoplus_{\alpha \in \Delta_\Sigma} \mathfrak{g}_\alpha.
\end{align*}
Moreover, we have  the corresponding triangular decomposition $\mathfrak{g}= \mathfrak{u}^{-}_\Sigma \oplus \mathfrak{l}_\Sigma \oplus \mathfrak{u}^{+}_\Sigma$. Note that if $\Sigma =\emptyset$ then $\mathfrak{p}_\Sigma(\pi) = \mathfrak{b}(\pi)$, and if $\Sigma = \pi$ then $\mathfrak{p}_\Sigma(\pi) = \mathfrak{g}$.

Recall that a $\mathfrak{g}$-module (respectively $\mathfrak{l}_\Sigma$-module)    $M$ is called  \emph{weight} if $\mathfrak{h}$ is diagonalizable on $M$.
 Let $V$ be a simple weight $\mathfrak{l}_\Sigma$-module. Set $\mathfrak{p}:=\mathfrak{p}_\Sigma(\pi)$ and consider $V$ as a $\mathfrak{p}$-module with trivial action of the nilradical $\mathfrak{u}^{+}_\Sigma$.
The \emph{generalized Verma} $\mathfrak{g}$-module $M^\mathfrak{g}_\mathfrak{p}(\Sigma,V)$ is the induced module 
\begin{align*}
  M^\mathfrak{g}_\mathfrak{p}(\Sigma, V) = {\mathop {\rm Ind}}^\mathfrak{g}_\mathfrak{p} V = U(\mathfrak{g}) \otimes_{U({\mathfrak{p}})} V .
\end{align*}

 \noindent
The module  $M^\mathfrak{g}_\mathfrak{p}(\Sigma, V)$ has a  unique maximal  submodule and a unique simple quotient $L^\mathfrak{g}_\mathfrak{p}(\Sigma, V)$. We write $M(\lambda)$ for the Verma module $M^\mathfrak{g}_{\mathfrak{b}(\pi)}(\emptyset, \mathbb{C}v_\lambda)$ and $L(\lambda)$ for $L^\mathfrak{g}_{\mathfrak{b}(\pi)}(\emptyset, \mathbb{C}v_\lambda)$ when it is clear which Borel subalgebra is meant. 

Let $M$ be a weight $\mathfrak{g}$-module.  For $\lambda\in \mathfrak h^*$ the subspace $M_{\lambda}$ of those $v\in V$ such that $hv=\lambda(h)v$ is the \emph{weight subspace} of weight $\lambda$. The dimension of $M_{\lambda}$ is the \emph{multiplicity} of weight $\lambda$. We say that a weight module is \emph{bounded} if all weight multiplicities are uniformly bounded. 
A weight $\mathfrak{g}$-module  $M$ is \emph{torsion free} if for any nonzero weight subspace $M_{\lambda}$ and any root $\alpha$, a nonzero root vector $X\in \mathfrak{g}_{\alpha}$ defines an isomorphism 
 between $M_{\lambda}$ and $M_{\lambda+\alpha}$. 
A weight module  $M$ is \emph{cuspidal} if $M$ is finitely generated torsion free module with finite  weight multiplicities.

\

\begin{Pro}\cite[Corollary 1.4]{Mat00}\label{cuspidal}
For a simple weight $\mathfrak{g}$-module $M$ with finite weight multiplicities, the following assertions are equivalent:

\begin{enumerate}[a)]
    \item $M$ is cuspidal;
    \item $M$ is torsion free;
    \item The support of $M$ is exactly one $Q$-coset.
\end{enumerate}
\end{Pro}

A weight module satisfying  Proposition \ref{cuspidal}, c) is called \emph{dense}.  

\

\subsection{Annihilators}

Recall that  $\lambda \in \mathfrak{h}^*$ is \emph{dominant} if $\left<\lambda+\rho_{\pi},\alpha^{\vee} \right> \notin\mathbb{Z}_{< 0}$, for all $\alpha \in \Delta_+$. Similarly $\lambda$ is \emph{antidominant} if $\left<\lambda+\rho_{\pi},\alpha^{\vee} \right> \notin\mathbb{Z}_{> 0}$ for all $\alpha \in \Delta_+$.
Also, $\lambda$ is \emph{regular}  if $\left<\lambda+\rho_{\pi},\alpha^{\vee} \right> \neq 0$ for all $\alpha \in \Delta$.   If $\mathcal{Z}$ is the center of $U(\mathfrak{g})$ and $\lambda\in\mathfrak{h}^{*}$,  then $\chi_{\lambda}:\mathcal{Z}\rightarrow \mathbb{C}$ stands for the central character of $L(\lambda)$.

For an associative algebra $A$ and an $A$-module $M$, the \emph{annihilator} of $M$ in $A$ is the ideal $\{a \in A\  | \ am = 0 \text{ for all } m \in M \}$ and will be denoted by $Ann_A (M) $.

\

 \begin{Pro}\cite[8.5.8]{Dix77}\label{max-ann}
If $\lambda\in \mathfrak{h}^*$ is dominant, then the annihilator of $L(\lambda)$ is the unique maximal two-sided ideal of $U(\mathfrak{g})$ containing   $U(\mathfrak{g})\ker \chi_{\lambda}$. 
 \end{Pro}
 
 \
 If $\lambda$ is regular dominant, then the correspondence $I\mapsto IM(\lambda)$ gives an order-preserving bijection between  two sided ideals of $U(\mathfrak{g})$ containing $U(\mathfrak{g})\ker \chi_{ \lambda}$  and submodules of $M(\lambda)$ \cite[Corollaries 4.3 and 4.8]{Jos79}.  
As a consequence of Proposition \ref{max-ann} we immediately have the following well known assertion.
 
 \begin{Co}\label{dom-ann}
 Let $\lambda, \mu \in \mathfrak{h}^*$ be dominant. If $ \mu=w \cdot \lambda$  for some $w\in W$ then  $\operatorname{Ann}_{U(\mathfrak{g})}L(\lambda)=\operatorname{Ann}_{U(\mathfrak{g})}L(\mu)$. Moreover, the converse holds if $\lambda$ is  regular.
 \end{Co}

\

\subsection{Affine vertex algebras} Consider  the non-twisted affine Kac-Moody algebra associated with
$\mathfrak{g}$:
$\widehat{\mathfrak{g}}=\mathfrak{g}[t,t^{-1}]+\mathbb C K$, where $K$ is a central element.
 Denote by $\widehat{\Delta}$  the set of roots of $\widehat{\mathfrak{g}}$. Then 
   $\widehat{\Delta}^{re}=\{\alpha+m\delta\mid \alpha\in \Delta,  m\in \mathbb{Z}\}$ is the set of real roots and 
   $\widehat{\Delta}^{re}_+=\{\alpha+m\delta\mid \alpha\in \Delta_+, m\in \mathbb{Z}_{\geq 0}\}\sqcup
   \{-\alpha+m\delta\mid \alpha\in \Delta_+, m\in \mathbb{Z}_{\geq 1}\}$ is the set of  positive real roots with respect to the basis $\pi\cup \{-\theta+\delta\}$, where $\theta$ is the maximal positive root of $\mathfrak{g}$ and $\delta$ is the minimal indivisible positive imaginary root.  
By $\widetilde{W}=W\ltimes Q^{\vee}$ we denote the extended affine Weyl group of $\widehat{\mathfrak{g}}$.

For $k\in \mathbb C$ denote by $V^k(\mathfrak{g})$ the \emph{universal affine vertex algebra}
associated with $\mathfrak{g}$ at  level $k$  (\cite{Kac98,FB04}):
\begin{align*}
 V^k(\mathfrak{g})=U(\widehat{\mathfrak{g}})\*_{U(\mathfrak{g}[t]\oplus \mathbb C K)}\mathbb C v_k,
\end{align*}
where
$\mathfrak{g}[t]v_k=0$  and $Kv_k=kv_k$.

The \emph{simple affine vertex algebra} associated with $\mathfrak{g}$
 at level $k$ is the unique simple graded quotient of $V^k(\mathfrak{g})$, which will be denoted by $V_k(\mathfrak{g})$. 
In the conformal case  (i.e. $k\neq -h^\vee$, where $h^\vee$ is the dual Coxeter number of $\mathfrak{g}$)  there is a one-to-one correspondence
between simple positive energy representations of 
 $V_k(\mathfrak{g})$ and simple  $A(V_k(\mathfrak{g}))$-modules   \cite{Zh}, where \begin{align*}
 A(V_k(\mathfrak{g}))\simeq U(\mathfrak{g})/I_k\label{eq:zhu-of-simple}
\end{align*}
is the \emph{Zhu's algebra} of $V_k(\mathfrak{g})$ and $I_k$ is some
two-sided ideal  of  $U(\mathfrak{g})$.
 
For $\lambda\in {\mathfrak h}^*$ and $k\in \mathbb C$ denote by $\widehat{L}_k( \lambda)$
 the simple highest weight  $\widehat{\mathfrak {g}}$-module with the highest weight 
$\widehat{\lambda}:=\lambda+k\Lambda_0\in \widehat{\mathfrak{h}}^*$, where $\Lambda_0(K)=1$, and $\Lambda_0(h)=0$ for $h\in \mathfrak h$.
Following \cite{KacWak89}, the module $\widehat{L}_k(\lambda)$, and the highest weight $\lambda$,  will be called \emph{admissible}  if
\begin{enumerate}
 \item $\langle \lambda+\widehat{\rho},\alpha^{\vee}\rangle \notin \mathbb{Z}_{\leq 0}$ for all $\alpha\in \widehat{\Delta}^{re}_+$;
 \item $\mathbb{Q}\widehat\Delta(\lambda)=\mathbb{Q} \widehat{\Delta}^{re}$,  
\end{enumerate}
where  $\widehat\Delta(\lambda)
=\{\alpha\in \widehat{\Delta}^{re}\ |\  \langle
 \lambda+\widehat{\rho},\alpha^{\vee}\rangle
\in \mathbb{Z}\}$. The  level $k$ of  $\widehat{\mathfrak {g}}$ is 
  \emph{admissible} 
  if the $\widehat{\mathfrak {g}}$-module
  $V_k(\mathfrak{g})\cong \widehat{L}_k(0)$ is admissible.

In this paper we are interested in type $A$ only, hence from now on we fix $n\geq 2$ and assume that $\mathfrak g$ is of type $A_n$.
In this case the description of admissible levels is as follows.
  
\begin{Pro}\cite[Proposition 1.2]{KacWak08}
\label{Pro:admissible number}
Let $\mathfrak{g}=\mathfrak{sl}_{n+1}$. The number  $k$ is admissible if and only if 
 \begin{align*}
  k+n=\frac{p}{q}-1
 \quad \text{with }p,q\in \mathbb{N},\ (p,q)=1,\ p\geq n+1.
 \end{align*}
\end{Pro}

\

\subsection{Relaxed highest weight modules}
Fix an admissible number $k$. 
Following \cite{AFR17}, we say that a $\mathfrak{g}$-module $M$ is {\it{admissible}} of  level $k$ if $M$ is an $A(V_k(\mathfrak{g}))$-module. The corresponding simple quotient of the induced module for $\widehat{\mathfrak {g}}$ is a {\it{relaxed highest weight module}}.
We have a one-to-one correspondence between the set of isomorphism classes of simple admissible $\mathfrak{g}$-modules of level $k$ and the isomorphism classes  of simple positive energy representations of $V_k(\mathfrak{g})$.

Denote by $Pr_k$ the set of admissible weights $\lambda$ such that there exists $y\in \widetilde{W}$ satisfying  $\widehat{\Delta}(\lambda)=y( \widehat{\Delta}(k\Lambda_0))$, and $Pr_{k,\mathbb{Z}}:=\{\lambda\in Pr_k\ |\  \lambda(K)=k,\ \langle \lambda,\alpha_i^{\vee}\rangle \in \mathbb{Z}\ \text{for all }i=1,\dots,l \}$, where $l$ is the rank of  $\mathfrak{g}$. If $k$  has the denominator $q$ such that   $(q,r^\vee)=1$, then (see \cite{KacWak08})
$$Pr_k=\mathop{\bigcup_{y\in \widetilde{W}}}_{ y(\widehat{\Delta}(k\Lambda_0)_+)\subset \widehat{\Delta}^{re}_+}Pr_{k,y},\quad \text{ and } \quad Pr_{k,y}:=y\cdot Pr_{k,\mathbb{Z}}.$$
Let $\overline{Pr_k}:=\{\bar \lambda\mid \lambda\in Pr_k\}\subset \mathfrak{h}^*,$ where 
$\bar \lambda$ is the projection of $\lambda$ to $\mathfrak{h}^*$. The importance of this set comes from the following Arakawa result of Arakawa. 

\begin{Th}\cite[Main Theorem]{Ara16}\label{Th:classification-of-admissible-modules}
  Let $k$ be admissible, $\lambda\in \mathfrak{h}^*$. Then $\widehat{L}_k(\lambda)$ is a module over $V_k(\mathfrak g)$ if and only if   $\lambda \in \overline{Pr_k}$. 
\end{Th}

For $\lambda\in \mathfrak{h}^*$ denote by 
$J_{\lambda}$  the corresponding primitive ideal $\operatorname{Ann}_{U(\mathfrak{g})}L(\lambda)$. 
By Theorem \ref{Th:classification-of-admissible-modules}, a simple
 $\mathfrak g$-module $M$  is  admissible  of level $k$ if and only if 
 $\operatorname{Ann}_{U(\mathfrak{g})}M=J_{\lambda}$ for some 
$ \lambda\in \overline{Pr_k}$.
As a consequence of  Proposition \ref{max-ann} and Corollary  \ref{dom-ann}, we have immediately the following statement.

 \begin{Pro} \cite[Proposition 2.4]{AFR17}\label{Pro:equiv-class-adm}
  For $\lambda \in \overline{Pr_k}$, the primitive ideal $J_{\lambda}$ is the unique maximal two-sided ideal of $U(\mathfrak{g})$ containing   $U(\mathfrak{g})\ker \chi_{\lambda}$. In particular, $J_{ \lambda}=J_{ \mu}$ for $\lambda,\mu \in \overline{Pr_k}$ if and only if  there exists $w\in W$ such that $ \mu=w\cdot\lambda$.
 \end{Pro}

\

\begin{Rem}
 The following characterization of the Zhu's algebra of $V_k(\mathfrak{g})$ was shown in    \cite[Theorem 3.4]{AE19}:  $$ A(V_k(\mathfrak{g}))\simeq  \prod_{\lambda \in  [\overline{{\rm Pr}}_k]} \dfrac{U(\mathfrak{g})}{J_\lambda}.$$
\end{Rem}

For a two-sided ideal $I$ of $U(\mathfrak g)$,
denote by $\operatorname{Var}(I)$  the \emph{associated  variety} of $I$ (the zero locus of  the associated graded ideal
 $gr I$ in $\mathfrak g^*$,
with respect to  the PBW filtration). If $I$ is primitive then
$\operatorname{Var} (I)=\overline{\mathbb{O}}$  for some nilpotent orbit $\mathbb{O}$ of $\mathfrak g$ \cite[Theorem 3.10]{Jos85}.

\begin{Def}
A simple $\mathfrak g$-module $M$ is \emph{in the orbit $\mathbb{O}$}  if
$\operatorname{Var} (\operatorname{Ann}_{U(\mathfrak g)}M)=\overline{\mathbb{O}}$.
 \end{Def}
 
 We have
 
 \begin{Th}\cite[Corollary 9.2 and Theorem 9.5]{Ara15}\label{Th;subsquence}
  A simple $\mathfrak{g}$-module $M$ in the orbit $\mathbb{O}$ is  admissible  of level $k$  if and only if $\operatorname{Ann}_{U(\mathfrak{g})}M=J_{\lambda}$ for some $\lambda\in [\overline{Pr}_k^{\mathbb{O}}]$. Moreover, there exists  a nilpotent orbit $\mathbb{O}_q$ that depends only on $q$  such that
  $$ \text{Var}(I_k)=\overline{\mathbb{O}_q}.$$
 \end{Th}

Hence,
 $\text{Var}(J_{\lambda})\subset \text{Var}(I_k)=\overline{\mathbb{O}_q}$
for any $\lambda\in \overline{Pr}_k$. Moreover, 
 by Theorems \ref{Th:classification-of-admissible-modules} and \ref{Th;subsquence}, we have
\begin{align*}
 \overline{Pr}_k=\bigsqcup_{\mathbb{O}\subset \overline{\mathbb{O}_q}}\overline{Pr}_k^{\mathbb{O}},
\end{align*}
where $\mathbb{O}$ is a nilpotent orbit of $\mathfrak{g}$ and
$
  \overline{Pr}_k^{\mathbb{O}}=\{\lambda\in \overline{Pr}_k\mid \text{Var}(J_{\lambda})=\overline{\mathbb{O}}\}.
 $

 Set $
[\overline{Pr}_k]=\overline{Pr}_k/\sim
$, where $\lambda\sim \mu$ if and only if there exists $w\in W$ such that $\mu=w\cdot \lambda$.
  By  Proposition \ref{Pro:equiv-class-adm}, the ideal
 $J_{\lambda}$ depends only on the class of $\lambda\in {\overline{Pr}_k}$ in $[\overline{Pr}_k]$.
 Therefore,
 \begin{align*}
  [\overline{Pr}_k]=\bigsqcup_{\mathbb{O}\subset \overline{\mathbb{O}_q}}[\overline{Pr}_k^{\mathbb{O}}],
 \end{align*}
 where $[\overline{Pr}_k^{\mathbb{O}}]$ is the image of $\overline{Pr}_k^{\mathbb{O}}$ in $  [\overline{Pr}_k]$.

\

\section{Twisted localization}\label{Section: Loc Functors}
In this section $F := \{ f_1,\dots ,f_r\}$ will denote any set of locally ad-nilpotent commuting elements of $U(\mathfrak{g})$,    ${\bf x}$ any element of ${\mathbb C}^r$. 
A $\mathfrak{g}$-module $M$ will be called \emph{$F$-bijective} if $f_i$ acts bijectively on $M$ for each $i=1,2,\dots, r$.
By $D_F U(\mathfrak{g})$ we denote the localization of $U(\mathfrak{g})$ relative to the multiplicative set $\langle F \rangle$ generated by $F $.  Similarly, for a $\mathfrak{g}$-module $M$ denote by $D_F M = D_F U(\mathfrak{g})\otimes_{U(\mathfrak{g})} M$  the \emph{localization of $M$ relative to $\langle F \rangle$}. We will consider $D_F M$ both as a $U(\mathfrak{g})$-module and as a $D_F U(\mathfrak{g})$-module.

\begin{Pro}\label{loc-max}
Let $M$ be an $F$-bijective $\mathfrak{g}$-module,  $L,N$  submodules of $M$, such that $L$ is not $F$-bijective, and $M/L$ is simple. If $L$ is  isomorphic to a proper submodule of $D_F N$ then $D_F N\simeq M$.
\end{Pro}

\begin{proof} As $D_F$ is an exact functor, then  $D_F N$ is a submodule of  $D_F M\simeq  M$.
 Since $L$ is a proper submodule of $D_F N$, the statement follows from the maximality of $L$ in $M$.
\end{proof}

For ${\bf x} = (x_1,\ldots,x_r) \in {\mathbb C}^r$ consider the automorphism $\Theta_F^{\bf x}$ of $D_F U(\mathfrak{g})$ such that  
$$
\Theta_F^{\bf x}(u):= \sum\limits_{
i_{1},\dots,i_r \geq 0} \binom{x_{1}}{i_{1}} \dots
\binom{x_r} {i_r} \,\mbox{ad}(f_{1})^{i_{1}}\dots
\mbox{ad}(f_r)^{i_r}(u) \,f_{1}^{-i_{1}}\dots
f_r^{-i_r},
$$
 for $u \in D_F U(\mathfrak{g})$, where $\binom{x}{i} :=x(x-1)...(x-i+1)/i!$ for $x \in {\mathbb C}$, $i \in {\mathbb Z}_{> 0}$ and $\binom{x}{0} :=1$ (\cite[Section 4]{Mat00}).

\begin{Def}
For a $D_F U(\mathfrak{g})$-module $N$,  we will denote by $\Phi^{\bf x}_{F} N$  the $D_F U(\mathfrak{g})$-module  on $N$ twisted by $\Theta_{F}^{\bf x}$, where   the new action is given by
 $$
u \cdot v^{\bf x} := ( \Theta_{F}^{\bf x}(u)\cdot v)^{\bf x},
$$
for $u \in D_F U(\mathfrak{g})$, $v \in N$. Here $v^{\bf x}$ stands for the element $v$ considered as an element of $\Phi^{\bf x}_{F} N$.
\end{Def}

If $M$ is a
 $\mathfrak{g}$-module  and ${\bf x} \in {\mathbb C}^r $, then  $D_{F}^{\bf x} M:=
\Phi^{\bf x}_{F} {D}_{F} M $ 
is
the {\it twisted localization of $M$
relative to $F$ and $\bf x$}.

\begin{Rem}
Note that for ${\bf x} \in {\mathbb Z}^r$,  we have
$\Theta_{F}^{\bf x}(u) = {\bf f}^{\mathbf x}u {\bf f}^{-\mathbf{x}}$, where ${\bf f}^{\bf x}:=f_{1}^{x_1}...f_r^{x_r} $. Moreover, $D_F^{\bf x} M $ and $D_F M$  are isomorphic for any $\mathfrak{g}$-module $M$.
\end{Rem}

\

The following important property of the twisted localization is a consequence of 
 \cite[Lemma 2.8]{GP18}.

\begin{Pro} \label{loc-ann}
Let $M$ be a weight $\mathfrak{g}$-module  with finite dimensional weight spaces and assume that $F$ is injective on $M$. Then 
 \begin{equation*}Ann_{U(\mathfrak{g})} M = Ann_{U(\mathfrak{g})} {D}_{F}  M\subset Ann_{U(\mathfrak{g})} {D}_{F}^{\bf x}  M,
 \end{equation*} 
 for any ${\bf x} \in {\mathbb C}^r $.
 Moreover,  $Ann_{U(\mathfrak{g})} M \subset Ann_{U(\mathfrak{g})} N,$ for any subquotient $N$ of  $D_{F}^{\bf x} M$.
\end{Pro}

 Applying Proposition \ref{loc-ann} and Proposition \ref{max-ann} to simple highest weight modules with dominant highest weights we obtain the following statement.

\begin{Co}\label{Localization annihilator} Let  $\lambda\in \mathfrak{h}^*$ be dominant,
${\bf x} \in {\mathbb C}^r $ and $F$ be injective on $L(\lambda)$.
If $N\neq 0$ is a simple subquotient of $\mathcal D_{F}^{\bf x} L(\lambda)$ then  
$$Ann_{U(\mathfrak{g})} L(\lambda)=Ann_{U(\mathfrak{g})} D_F^{\bf x}  L(\lambda)=Ann_{U(\mathfrak{g})} N .$$
\end{Co}

 \section{Relation Gelfand--Tsetlin modules}\label{section:GT modules} 

 \subsection{Gelfand--Tsetlin modules}
  For any flag $\mathcal F: \mathfrak{gl}_1\subset \cdots \subset \mathfrak{gl}_{n+1}$ we have an induced  flag 
 $U_1\subset \cdots \subset U_{n+1} $ of the  corresponding  universal enveloping algebras.
Let $\mathcal{Z}_{m}$ be the center of $U_{m}$. 
 Following \cite{DFO94}, we call the subalgebra $\Gamma_{\mathcal{F}}$ of  $U:=U_{n+1}$ generated by $\{\mathcal{Z}_m\ |\ m=1,\ldots, n+1 \}$ the \emph{Gelfand--Tsetlin subalgebra} of $U$ (with respect to $\mathcal{F}$).  If the flag is given by left-upper corner inclusions, the corresponding Gelfand--Tsetlin subalgebra will be called the \emph{standard Gelfand--Tsetlin subalgebra} and will be denoted by $\Gamma_{st}$. 
Consider a standard basis $E_{ij}$, $i,j=1, \ldots n+1$ of $\mathfrak{gl}_{n+1}$. Then the standard flag $\mathcal F_{st}$ of $\mathfrak{gl}_{n+1}$
consists of $\mathfrak{gl}_k$, $k=1, \ldots, n+1$, with $\mathfrak{gl}_s$ generated by $E_{ij}$, $i,j=1, \ldots s$ for all $s$.  If $w\in W$ and $\mathcal F=w\mathcal F_{st}$ then $\Gamma_{\mathcal F}=w\Gamma_{st}$ with a natural action of $w$.

\begin{Def}\label{definition-of-GZ-modules} 
 A finitely generated $U$-module $M$ is called a \emph{$\Gamma_{\mathcal{F}}$-Gelfand--Tsetlin module} if $M$ splits into  a direct sum of $\Gamma_{\mathcal{F}}$-modules:
 $$M=\bigoplus_{\mathsf{m}\in \emph{Specm}\Gamma_{\mathcal{F}}}M(\mathsf{m}),$$
 where $$M(\mathsf{m})=\{v\in M\ |\ \mathsf{m}^{k}v=0 \text{ for some }k\geq 0\}.$$
 \end{Def}

Identifying $\mathsf{m}$ with the homomorphism $\chi:\Gamma_{\mathcal{F}} \rightarrow {\mathbb C}$ with $\text{Ker} \chi=\mathsf{m}$, we will call $\mathsf{m}$ a \textit{Gelfand--Tsetlin character} of $M$ if $ M(\mathsf{m}) \neq 0$. 
The dimension $\dim M(\mathsf{m})$ will be called the \emph{Gelfand--Tsetlin multiplicity of $\mathsf{m}$}.

\begin{Rem}
 Let $\tau: \mathfrak{gl}_{n+1}\rightarrow \mathfrak{sl}_{n+1}$ be a natural projection which extends to a homomorphism $\bar{\tau}: U(\mathfrak{gl}_{n+1})\rightarrow U(\mathfrak{sl}_{n+1})$. If $\Gamma$ is a Gelfand-Tsetlin subalgebra of $\mathfrak{gl}_{n+1}$ then  $\bar{\tau}(\Gamma)$ of $\Gamma$ is called a Gelfand-Tsetlin subalgebra of $\mathfrak{sl}_{n+1}$.
\end{Rem}

Unless otherwise is stated, from now on we will assume that  $\Gamma_{\mathcal F}=\Gamma_{st}$ which we simply denote by $\Gamma$. We will refer to $\Gamma$-Gelfand--Tsetlin modules simply as Gelfand--Tsetlin modules.

\subsection{Relation modules}\label{Section: Rel Modules}

A class of \emph{relation}  Gelfand--Tsetlin modules was constructed in \cite{FRZ19}. These modules have  properties similar to finite-dimensional modules \cite{GT50} and to generic Gelfand--Tsetlin modules \cite{DFO94}.  We recall their construction here since it will be used later on.

Set $\mathfrak{V}:=\{(i,j)\ |\ 1\leq j\leq i\leq n+1\}$, and  $\mathcal{R}:=\mathcal{R}^{-}\cup\mathcal{R}^{0}\cup\mathcal{R}^{+}\subset \mathfrak{V}\times\mathfrak{V}$, where  \begin{align*}
 \mathcal{R}^+ &:=\{((i,j);(i-1,t))\ |\ 2\leq j\leq i\leq n+1,\ 1\leq t\leq i-1\}\\
 \mathcal{R}^- &:=\{((i,j);(i+1,s))\ |\ 1\leq j\leq i\leq n,\ 1\leq s\leq i+1\}\\
 \mathcal{R}^{0}&:=\{((n+1,i);(n+1,j))\ |\ 1\leq i\neq j\leq n+1\}
 \end{align*}

Any subset $\mathcal{C}\subseteq \mathcal{R}$ will be called a \emph{set of relations}. 
With  any $\mathcal{C}\subseteq \mathcal{R}$ we associate a directed graph $G(\mathcal{C})$ with the set of vertices  $\mathfrak{V}$, that has an arrow  from vertex $(i,j)$ to $(r,s)$ if and only if $((i,j);(r,s))\in\mathcal{C}$. For convenience, we will picture the set $\mathfrak{V}$ as a triangular tableau with $n+1$ rows, where the $k$-th row is $\{(k,1), \ldots, (k,k)\}$, $k=1, \ldots, n+1$.

For  $v\in \mathbb{C}^{\frac{(n+1)(n+2)}{2}}$ denote by $T(v)$  the image of $v$ via the natural isomorphism between $\mathbb{C}^\frac{(n+1)(n+2)}{2}$ and $\mathbb{C}^{n+1}\times\cdots\times\mathbb{C}^{1}$. If $T(v)=(v^{(n+1)},\ldots,v^{(1)})$ then we  refer to $v^{(k)}=(v_{k1},\ldots,v_{kk})$ as the $k$-th row of $T(v)$. 
Hence, we can picture $T(v)$ as a triangular tableau,  a \textit{Gelfand-Tsetlin tableau} of height $n+1$. Finally, by ${\mathbb Z}_0^\frac{(n+1)(n+2)}{2}$ we will denote the set of vectors $v$ in $\mathbb{Z}^\frac{(n+1)(n+2)}{2}$ such that $v^{(n+1)}={\bf 0}$.

\begin{Def} \cite[Definition 4.2]{FRZ19}
Let $\mathcal{C}$ be a set of relations and $T(L)$ any Gelfand-Tsetlin tableau, where $L=(l_{ij})\in \mathbb{C}^{\frac{(n+1)(n+2)}{2}}$.

\begin{enumerate}[a)]
\item We  say that \emph{$T(L)$ satisfies $\mathcal{C}$} if:
\begin{enumerate}[(i)]
\item $l_{ij}-l_{rs}\in \mathbb{Z}_{\geq 0}$ for any $((i,j); (r,s))\in \mathcal{C}\cap (\mathcal{R}^+\cup \mathcal{R}^0)$;
\item $l_{ij}-l_{rs}\in \mathbb{Z}_{> 0}$ for any $((i,j); (r,s))\in \mathcal{C}\cap \mathcal{R}^-$.
\end{enumerate}
\item We say that \emph{$T(L)$ is a $\mathcal{C}$-realization} if $T(L)$ satisfies $\mathcal{C}$ and for any $1\leq k\leq n$ we have, $l_{ki}-l_{kj}\in \mathbb{Z} $ if only if $(k,i)$ and $(k,j)$ in the same connected component of the undirected graph associated with $G(\mathcal{C})$.
 \item We call $\mathcal{C}$ \emph{noncritical} if for any $\mathcal{C}$-realization $T(L)$ one has $l_{ki}\neq l_{kj}$, $1\leq k\leq n,\ i\neq j$, if $(k,i)$ and $(k,j)$ are in the same connected component of $G(\mathcal{C})$.
\item Suppose that $T(L)$ satisfies $\mathcal{C}$. Then ${\mathcal B}_{\mathcal{C}}(T(L))$  denotes the set of all tableaux of the form $T(L+z)$, $z\in {\mathbb Z}_0^\frac{(n+1)(n+2)}{2}$ satisfying $\mathcal{C}$. Also, $V_{\mathcal{C}}(T(L))$  denotes the complex vector space spanned by ${\mathcal B}_{\mathcal{C}}(T(L))$.
\end{enumerate}
\end{Def}

\

\begin{Rem}
Note that if $T(L)$ satisfies $\mathcal{C}$, and $\mathcal{C}$ is the maximal set of relations satisfied by $T(L)$, then $T(L)$ is a $\mathcal{C}$-realization.
\end{Rem}

\begin{Def}  \cite[Definition 4.4]{FRZ19}
Let $\mathcal{C}$ be a set of relations. We call  $\mathcal{C}$  \emph{admissible} if for any $\mathcal{C}$-realization $T(L)$,  $V_{\mathcal{C}}(T(L))$ has a structure of a $\mathfrak{sl}_{n+1}$-module, endowed with the following action of the generators of $\mathfrak{sl}_{n+1}$:

\begin{equation}\label{Gelfand--Tsetlin formulas}
\begin{split}
E_{k,k+1}(T(w))=-\sum_{i=1}^{k}\left(\frac{\prod_{j=1}^{k+1}( w_{ki}- w_{k+1,j})}{\prod_{j\neq i}^{k}( w_{ki}- w_{kj})}\right)T(w+\delta^{ki}),\\
E_{k+1,k}(T(w))=\sum_{i=1}^{k}\left(\frac{\prod_{j=1}^{k-1}( w_{ki}- w_{k-1,j})}{\prod_{j\neq i}^{k}( w_{ki}- w_{kj})}\right)T(w-\delta^{ki}),\\
H_k(T(w))=\left(2\sum_{i=1}^{k} w_{ki}-\sum_{i=1}^{k-1} w_{k-1,i}-\sum_{i=1}^{k+1} w_{k+1,i}-1\right)T(w).
\end{split}
\end{equation}
 \end{Def}

A pair  $\{(k,i), (k,j)\} \subset \mathfrak{V}$ is an \emph{adjoining pair} for a graph $G$ if $i<j$, there is a path in $G(\mathcal{C})$ from $(k,i)$ to $(k,j)$, and there is no path in $G(\mathcal{C})$ from $(k,i)$ to $(k,j)$ passing trough $(k,t)$ with $i<t<j$.\\

Suppose that $\mathcal{C}$ is a noncritical set of relations whose associated graph $G = G(\mathcal{C})$ satisfies the following conditions: 
\begin{itemize}
\item[(i)]  $G$ is reduced;
\item[(ii)] If there is a path in $G$ connecting $(k,i)$ and $(k,j)$ with tail $(k,i)$ and head $(k,j)$, then $i<j$ (in particular $G$ does not contain loops);
\item[(iii)] If  the graph $G$ contains an arrow connecting $(k,i)$ and $(k+1,t)$  then  $(k+1,s)$ and $(k,j)$  with $i<j$, $s<t$ are not connected in $G$.
\end{itemize}

By \cite[Theorem 4.33]{FRZ19}, the set $\mathcal{C}$  is admissible if and only if
$G$ is a union of connected graphs  satisfying the following 

\

\noindent
\emph{\bf $\Diamond$-Condition}: 
For every adjoining pair $\{(k,i),(k,j)\}$, $1\leq k\leq n$, there exist $p, q$ such that $\mathcal{C}_{1}\subseteq\mathcal{C}$ or, there exist $s<t$ such that $\mathcal{C}_{2}\subseteq\mathcal{C}$, where the graphs associated to $\mathcal{C}_{1}$ and $\mathcal{C}_{2}$ are as follows
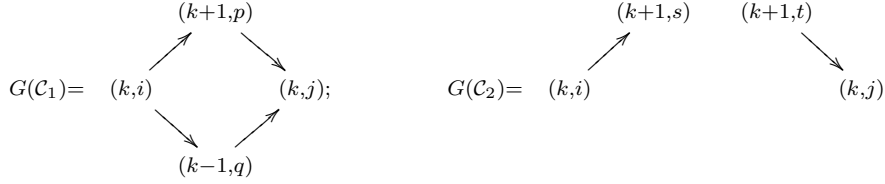
\begin{equation*}
\begin{tabular}{c c c c}
\xymatrixrowsep{0.5cm}
\xymatrixcolsep{0.1cm}
\xymatrix @C=0.2em{
  &   &\scriptstyle{(k+1,p)}\ar@*{}[rd]   &   & \\
 \scriptstyle{G(\mathcal{C}_{1})=}  &\scriptstyle{(k,i)}\ar@*{}[rd] \ar@*{}[ru]  &    &\scriptstyle{(k,j)};   &  \\
   &   &\scriptstyle{(k-1,q)}\ar@*{}[ru]   &   & }
&\ \ &
\xymatrixrowsep{0.5cm}
\xymatrixcolsep{0.1cm}\xymatrix @C=0.2em {
   &   &\scriptstyle{(k+1,s)}    &   &\scriptstyle{(k+1,t)}\ar@*{}[rd]&& \\
  \scriptstyle{G(\mathcal{C}_{2})=} &\scriptstyle{(k,i)} \ar@*{}[ru]  & &   & & \scriptstyle{(k,j)} \\
   &   &   &   & &&}
\end{tabular}
\end{equation*}

From now on we will assume that $\mathcal{C}$ is an admissible  set of relations and consider the  $\mathfrak{sl}_{n+1}$-module
$V_{\mathcal{C}}(T(L))$. The simplicity criteria for $V_{\mathcal{C}}(T(L))$ is as follows.

 \begin{Th}\cite[Theorem 5.6]{FRZ19}\label{thm-irr}
 The Gelfand-Tsetlin module $V_{\mathcal{C}}(T(L))$ is simple if and only if $\mathcal{C}$ is the maximal (admissible) set of relations satisfied by $T(L)$.
 \end{Th}

We will call $V_{\mathcal{C}}(T(L))$  a \emph{$\Gamma$-relation Gelfand--Tsetlin module}. Similar construction can be made for any Gelfand--Tsetlin subalgebra $\Gamma_{\mathcal F}=w\Gamma_{st}$, $w\in W$. Applying $w$ to the formulas \eqref{Gelfand--Tsetlin formulas} we obtain explicit action of $\mathfrak{sl}_{n+1}$ on $\Gamma_{\mathcal F}$-relation Gelfand--Tsetlin modules.

\subsection{Classification of highest weight relation modules}
The goal of this section is to describe all simple highest weight modules that can be realized as $V_{\mathcal{C}}(T(L))$ for some admissible set of relations $\mathcal{C}$. 
From now on we fix $\mathfrak g:=\mathfrak{sl}_{n+1}$,  the set of simple roots $\pi=\{\alpha_1, \ldots, \alpha_{n}\}$, and set $\alpha_{r,s}:=\alpha_{r}+\ldots+\alpha_{s}$ for $1\leq r\leq s\leq n$. We use elements $E_{ij}$, $i,j=1, \ldots, n+1$, $i\neq j$, $H_{k}:=E_{kk}-E_{k+1,k+1}$, $k=1, \ldots, n$, as a basis of  $\mathfrak g$.

From  \cite[Proposition 5.9]{FRZ19}, we have  the following statement. 

\begin{Lem}\label{Lem: hw module} 
If $\left<\lambda+\rho_{\pi},\alpha^{\vee} \right> \notin\mathbb{Z}_{\leq 0}$ for all $\alpha \in \Delta_+ \setminus \{ \alpha_{k,n}\ | \ k=1, \dots , n\} $, then the simple highest weight module $L(\lambda)$  is a $\Gamma$-relation module.
\end{Lem}


The following assertion  classifies all simple highest weight $\Gamma$-relation modules.

\begin{Th}
\label{hw module}
The simple highest weight module $L(\lambda)$ is a $\Gamma$-relation module if and only if one of the following conditions holds:
\begin{enumerate}[a)]
\item $\left<\lambda+\rho_{\pi},\alpha^{\vee} \right> \notin\mathbb{Z}_{\leq 0}$, for all $\alpha \in \Delta_+ \setminus \{ \alpha_{k,n}\ | \ k=1, \dots , n\} $.
\item There exist unique $i,j$ with $1\leq i\leq j< n$ such that:
    \begin{enumerate}[i)]
    \item $\left<\lambda+\rho_{\pi},\alpha_k^{\vee} \right> \in\mathbb{Z}_{> 0}$ for each $k >j$,
    \item$\left<\lambda+\rho_{\pi},\alpha^{\vee} \right> \notin\mathbb{Z}_{\leq 0}$ for all $\alpha \in \Delta_+ \setminus \{ \alpha_{i,k} \ | \ k\geq j\}$,
    \item $\left<\lambda+\rho_{\pi},\alpha_{i,n}^{\vee} \right> \in\mathbb{Z}_{\leq 0}$. 
     \end{enumerate}
\end{enumerate}
\end{Th}

\begin{proof}
Suppose that $L(\lambda)$ is a $\Gamma$-relation  Gelfand--Tsetlin module. Then $L(\lambda) \simeq V_{\mathcal{C}}(T(v))$ for some tableau $T(v)$. As $L(\lambda)$ is simple, we can assume without loss of generality that $T(v)$ is a highest weight vector of weight $\lambda$ and $\mathcal{C}$ is the maximal set of relations satisfied by $T(v)$ (cf. Theorem \ref{thm-irr}). Then $E_{k,k+1}(T(v))=0$ for all $k=1, \dots, n$, which implies  $v_{ij}=v_{kj}$ for all $1\leq i,k<j\leq n+1$ (see (\ref{Gelfand--Tsetlin formulas})), and $\mathcal{C} \supset \{((i+1,j);(i,j))\ |\ 1\leq j\leq i\leq n\} $. Set $v_{j}:=v_{n+1,j}$, and note that $v_i-v_{j+1}=\left<\lambda+\rho_{\pi},\alpha_{i,j}^{\vee} \right>$ for all $1\leq i\leq j\leq n$. Now suppose that for some  $\alpha \in \Delta_+ \setminus \{ \alpha_{k,n}\ | \ k=1, \dots , n\} $ we have  $\left<\lambda+\rho_{\pi},\alpha^{\vee} \right> \in\mathbb{Z}_{\leq 0}$. Hence, there exists a pair $(r,s)$ with $1\leq r\leq s<n$ such that $v_r-v_{s+1}=\left<\lambda+\rho_{\pi},\alpha_{r,s}^{\vee} \right> \in \mathbb{Z}_{<0}$. This shows that $I=\{(r,s)\ |\ v_s-v_r \in \mathbb{Z}_{>0} \mbox{ and } 1\leq r< s<n    \}\neq \emptyset$. Now choose  $i=min\{r \ |\ (r,s)\in I \} $ and $j=min\{s\ | \ (i,s) \in I \}$. Then  $(j,i)$ and $(j,j)$ form an adjoining pair and the associated graph  looks as follows: 

$$
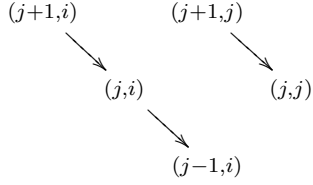
\xymatrix @R=0.5cm @C=0.1cm {\scriptstyle{(j+1,i)}\ar[dr] & & \scriptstyle{(j+1,j)} \ar[dr] & \\ 
&\scriptstyle{(j,i)} \ar[dr]& & \scriptstyle{(j,j)}  \\ 
 & & \scriptstyle{(j-1,i)} & }
$$
On the other hand, $\mathcal{C}$ is an admissible set of relations and, hence, by the $\Diamond$-Condition the associated graph must satisfy 
$$
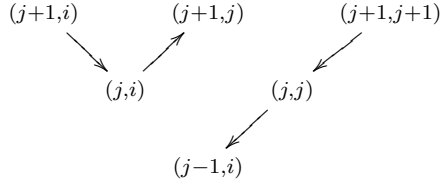
\xymatrix @R=0.5cm @C=0.1cm {\scriptstyle{(j+1,i)}\ar[dr] & & \scriptstyle{(j+1,j)} & & \scriptstyle{(j+1,j+1)} \ar[dl] \\ 
&\scriptstyle{(j,i)} \ar[ur]& & \scriptstyle{(j,j)} \ar[dl] \\ 
 & & \scriptstyle{(j-1,i)} & }
$$
Therefore, $v_j-v_{j+1} \in\mathbb{Z}_{>0}$ and $v_{j+1}-v_i \in\mathbb{Z}_{>0}$. 
Repeating the same argument, we conclude that  $v_k-v_{k+1} \in\mathbb{Z}_{>0}$ for all $k=j,\dots, n$ and $v_{n+1}-v_i \in\mathbb{Z}_{\geq 0}$. Consequently we have  $v_r-v_s \in \mathbb{Z}_{>0}$ for all $j\leq r <s \leq n+1$. On the other hand, from the choice of $i$ we have that $v_r-v_s \notin \mathbb{Z}_{\leq 0}$ for each $1\leq r<i$ and  $s>r$. From the  definition of $j$ we also have  $v_i-v_s \notin \mathbb{Z}_{\leq 0}$ for each $i<s<j$. Given that $v_s-v_i \in \mathbb{Z}_{\geq 0}$ for all $s \geq j$ and $v_i-v_r \notin \mathbb{Z}_{\leq 0}$ for all $i<r<j$, we obtain $v_s-v_r \notin \mathbb{Z}_{\leq 0}$ for all $s, r$ such that $i<r<j\leq s\leq n+1$. Finally, $v_s-v_r \notin \mathbb{Z}_{\leq 0}$ for all $s, r$ such that $i<r<s<j$. Indeed, assume that $v_s-v_r \in \mathbb{Z}_{< 0}$ for some $i<r<s<j$. Then there exists $r<j'\leq s$ such that $v_k-v_{k+1} \in\mathbb{Z}_{>0}$ for all $k=j',\dots, n$. In particular, $v_s-v_j \in\mathbb{Z}_{>0}$, which is a contradiction. 

Conversely, suppose first that  a) holds. Then by Lemma \ref{Lem: hw module} we have that $L(\lambda)$ is a $\Gamma$-relation module. Now assume b) and let $v_s-v_{s+1}= \left<\lambda+\rho_{\pi},\alpha_s^{\vee} \right> $ for each $1\leq s\leq n$ such that  $\displaystyle \sum_{s=1}^{n+1}v_s=-\binom{n+1}{2}$. Then for some $1\leq i\leq j<n$ the following conditions are satisfied: 
\begin{itemize}
    \item $v_{n+1}-v_i\in \mathbb{Z}_{\geq 0}$,
    \item $v_r-v_s\in \mathbb{Z}_{> 0}$ for all $j+1\leq r <s \leq n+1$,
     \item $v_r-v_s\notin \mathbb{Z}_{\leq 0}$ for all $1\leq r \leq j$,  $r< s  \leq n+1$ and $r,s\neq i$,
     \item $v_r-v_i\in \mathbb{Z}_{>0}$ for all $j+1\leq r \leq n$,
     \item $v_r-v_i\notin \mathbb{Z}_{\leq 0}$ for all $1\leq r \leq j$ and $r\neq i$.
\end{itemize}

Let $T(v)$ be a Gelfand--Tsetlin tableau with entries
\begin{align*}
v_{rs}=\begin{cases}
v_s, & \ \ \text{if}\ \  1 \leq s <i \text{ or } i\leq r \leq j,\\
v_{s+j-i+1}, & \ \ \text{if}\ \  i\leq s < r+i-j,\\
v_{s-r+j}, & \ \ \text{if}\ \ s\geq r+i-j,
\end{cases}
\end{align*}

for $1\leq s \leq r \leq n+1$, and let $\mathcal{C}$ be the maximal set of relations satisfied by $T(v)$. To prove that $\mathcal{C}$ is admissible we  consider the following cases.

\textbf{Case I:}    Suppose that $(k,r)$ and $(k,s)$ form an adjoining pair for  some $1\leq k<j$. Then we have an indecomposable subset $\mathcal{C}'$ of $\mathcal{C}$ with the  associated graph 

$$
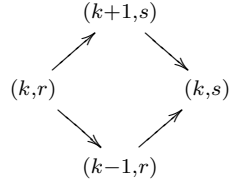
\xymatrix @R=0.5cm @C=0.1cm {  & \scriptstyle{(k+1,s)} \ar[dr] & \\ 
\scriptstyle{(k,r)} \ar[dr] \ar[ur]& & \scriptstyle{(k,s)}  \\ 
  & \scriptstyle{(k-1,r)}  \ar[ur]& }
$$
The same happens in the following cases:
\begin{itemize}
    \item  $j\leq k\leq n$ and $1< s <i$;
    \item  $j< k\leq n$ and $1\leq r < s=i$;
    \item $j+1< k\leq n$ and $i< s <n+i-j$ where  $r=s-1$.
\end{itemize}

\textbf{Case II:} Suppose $j\leq k\leq n$, $1\leq r < i$ and $s \geq k+i-j$. If  $(k,r)$ and $(k,s)$ form an adjoining pair, then we have an indecomposable subset $\mathcal{C}'$ of $\mathcal{C}$ with the following associated graph:
  $$
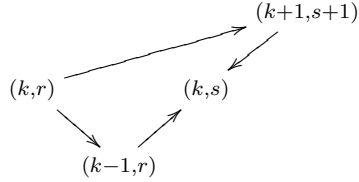
\xymatrix @R=0.5cm @C=0.1cm {  &  &  & \scriptstyle{(k+1,s+1)} \ar[dl] \\ 
\scriptstyle{(k,r)} \ar[dr] \ar[urrr]& & \scriptstyle{(k,s)}  \\ 
  & \scriptstyle{(k-1,r)}  \ar[ur]& }
$$
The graph is the same in the case $k=j$ and $r \geq i $.

\textbf{Case III:}  Fix $j< k\leq n$ with $r \geq k+i-j$. If  $(k,r)$ and $(k,s)$ form an  adjoining pair, then the associated graph  of the indecomposable subset $\mathcal{C}'$ of $\mathcal{C}$ is as follows: 
$$
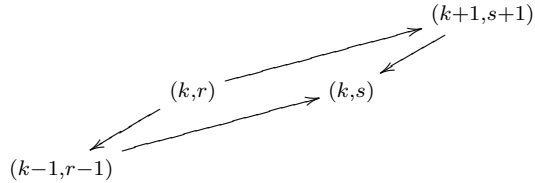
\xymatrix @R=0.5cm @C=0.5cm {   & &   &  & \scriptstyle{(k+1,s+1)} \ar[dl] \\ 
&\scriptstyle{(k,r)} \ar[dl] \ar[urrr] & & \scriptstyle{(k,s)} & \\ 
\scriptstyle{(k-1,r-1)}  \ar[urrr]  & &  &}
$$

\textbf{Case IV:}  Finally, let $j< k\leq n$,  $s = k+i-j$ and $r=s-1$. In this case $(k,r)$ and $(k,s)$ is an adjoining pair with the associated graph  as follows: 
$$
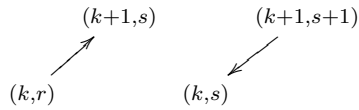
\xymatrix @R=0.5cm @C=0.1cm {  & \scriptstyle{(k+1,s)}  &  & \scriptstyle{(k+1,s+1)} \ar[dl] \\ 
\scriptstyle{(k,r)}  \ar[ur]& & \scriptstyle{(k,s)} }
$$
The same happens in the case: $j< k\leq n$ and $r = k+i-j-1$ for each adjoining pair $\{(k,r),(k,s)\}$. 

Therefore $\mathcal{C}$ is admissible and $V_{\mathcal{C}}(T(v))$ is a simple module by Theorem \ref{thm-irr}. Hence,  $V_{\mathcal{C}}(T(v))=U(\mathfrak{g})T(v)$, $E_{k,k+1}(T(v))=0$ and $H_k(T(v))=\left<\lambda,\alpha_k^{\vee} \right> T(v)$ for all $k=1, \dots, n$. 

\end{proof}

In particular, applying Theorem \ref{hw module} and   \cite[Proposition 8.5]{Mat00} we get the following criterion for infinite-dimensional simple highest weight modules to be bounded $\Gamma$-relation modules.

\begin{Co} 
\label{prop:relationbounded}
The simple  highest weight module $L(\lambda)$  is a bounded infinite-dimensional $\Gamma$-relation module if and only if one of the following conditions holds:
\begin{enumerate}[a)]
\item $\left<\lambda+\rho_{\pi},\alpha_n^{\vee} \right> \notin\mathbb{Z}_{> 0}$,  and $\left<\lambda+\rho_{\pi},\alpha_k^{\vee} \right> \in\mathbb{Z}_{> 0}$ for all $k< n$;
\item $\left<\lambda+\rho_{\pi},\alpha_1^{\vee} \right> \notin\mathbb{Z}$,  and $\left<\lambda+\rho_{\pi},\alpha_k^{\vee} \right> \in\mathbb{Z}_{> 0}$ for all $k> 1$;  
\item    $\left<\lambda+\rho_{\pi},\alpha_1^{\vee} \right> \in\mathbb{Z}_{< 0}$, $\left<\lambda+\rho_{\pi},\alpha_{1n}^{\vee} \right> \in\mathbb{Z}_{\leq 0}$, and $\left<\lambda+\rho_{\pi},\alpha_k^{\vee} \right> \in\mathbb{Z}_{> 0}$ for all $k> 1$;
\item There exists a unique $i\in \{2,\dots ,n-1 \}$ such that $\left<\lambda+\rho_{\pi},\alpha_i^{\vee} \right> \in\mathbb{Z}_{< 0}$, $\left<\lambda+\rho_{\pi},\alpha_{i-1,i}^{\vee}\right> \in\mathbb{Z}_{>0}$, $\left<\lambda+\rho_{\pi},\alpha_{in}^{\vee}\right> \in\mathbb{Z}_{\leq 0}$, and $\left<\lambda+\rho_{\pi},\alpha_k^{\vee} \right> \in\mathbb{Z}_{> 0}$ for all  $k\neq i$;
\item  There exists a unique $i\neq n$ such that $\left<\lambda+\rho,\alpha_i^{\vee} \right> \notin\mathbb{Z}$, $\left<\lambda+\rho_{\pi},\alpha_{i+1}^{\vee} \right> \notin\mathbb{Z}$, $\left<\lambda+\rho,\alpha_{i,i+1}^{\vee}\right> \in\mathbb{Z}_{>0}$, and $\left<\lambda+\rho_{\pi},\alpha_k^{\vee} \right> \in\mathbb{Z}_{> 0}$ for all $k\neq i, i+1$.
\end{enumerate}
\end{Co}

\begin{Rem} For  $i=1, 2, \dots, n$, the highest weight module $L(-\omega_i)$ is bounded but  not a $\Gamma$-relation module. 
\end{Rem}

\begin{Co} \label{verma module}
Let $\lambda \in \mathfrak{h}^*$. The following conditions are equivalent
\begin{enumerate}[a)]
\item The Verma module $M(\lambda)$ is a simple $\Gamma$-relation module; 
\item $\left<\lambda+\rho_{\pi},\alpha^{\vee} \right> \notin\mathbb{Z}$, for any $\alpha \in \Delta_+ \setminus\{ \alpha_{k,n}\ |\  k=1, \dots, n\}$, and $\left<\lambda+\rho_{\pi},\alpha_{k,n}^{\vee} \right> \notin\mathbb{Z}_{> 0}$  for  $k=1, \dots, n$.
\end{enumerate}
\end{Co}

\begin{Rem}
By Corollary \ref{verma module}, $M(\lambda)$ is a simple $\Gamma$-relation module if and only if $M(\lambda) \simeq V_{\mathcal{C}}(T(v))$, where $T(v)$ is a generic tableau and  $\mathcal{C}$ is the maximal set of relations satisfied by $T(v)$. This generalizes   \cite[Proposition 1]{Maz98} (cf. \cite[Corollary 3.7]{AFR17} and  \cite[Example 5.10]{FRZ19}). 
\end{Rem}

  Since $L(\lambda)$ has finite-dimensional weight subspaces, then  it is 
 a $w\Gamma$-Gelfand--Tsetlin module for any  $w\in W$. 
 Denote by $\widetilde{W}$ the extension of the Weyl group of $\mathfrak g$ by the symmetries of the root system.  
 Let $w\in \widetilde{W}$ and $\mathfrak b=w\mathfrak b_{st}$. Then, clearly,  $L_{\mathfrak b}(w\lambda)$ is a $w\Gamma$-relation Gelfand--Tsetlin module if and only if $w\lambda$ satisfies the conditions of Corollary \ref{prop:relationbounded} for 
 $w\alpha_k^{\vee}$, $k=1, \ldots, n$.  In particular, we obtain 
 
 \begin{Co}\label{cor-outras Borel} Let $\lambda\in \mathfrak{h}^*$ and  $\left<\lambda+\rho_{\pi},\beta_j^{\vee} \right> \in\mathbb{Z}_{> 0}$ for some simple roots $\beta_j$, $j=1, \ldots, t $. Consider the minimal $w\in W$  such that  $\left<w(\lambda+\rho_{\pi}), \beta_j^{\vee} \right> \in\mathbb{Z}_{\leq 0}$ for all $j=1, \ldots, t$, and 
 set $\Gamma_{\mathcal F}=w\Gamma$. Then 
 the  simple highest weight module $L(\lambda)$  is a bounded infinite-dimensional  $\Gamma_{\mathcal F}$-relation Gelfand--Tsetlin module if and only if $\lambda$ satisfies one of the conditions of Corollary \ref{prop:relationbounded}.

 \end{Co}
 
 \begin{proof}
 Since $L_{w\mathfrak b_{st}}(w\lambda)\simeq L(\lambda)$,
 it is sufficient to consider the weight $w\lambda$ and apply Corollary \ref{prop:relationbounded}.
\end{proof}

Corollary \ref{cor-outras Borel}  shows that  $L(\lambda)$ can have different realizations via Gelfand--Tsetlin tableaux as a relation module.

\begin{Lem}\label{lem-key}
Suppose $\lambda$ satisfies conditions a) or e) of Corollary \ref{prop:relationbounded}. Let $i>1$ and set $w_i=s_{i-1} s_i s_{i-2} s_{i-1} \cdots  
s_2s_3s_1s_2 \in W$, then  $L_i=L_{w_i^{-1}b_{st}}(w_i^{-1}\lambda)$ is a $\Gamma$-relation  Gelfand-Tsetlin module.

\end{Lem}

\begin{proof} 

Let us consider $\{v_{1},\ldots,v_{n+1}\}\subseteq \mathbb{C}$ such  that  $ v_1-v_2=\left <\lambda + \rho_{\pi}, \alpha_i^{\vee} \right> $,  $v_1- v_{i+1}=-\left <\lambda + \rho_{\pi}, \alpha_{i-1}^{\vee} \right>$, $ v_2-v_{i+2}=\left <\lambda + \rho_{\pi}, \alpha_{i+1}^{\vee} \right> $, $ v_{k+2}-v_{k+3}=\left <\lambda + \rho_{\pi}, \alpha_k^{\vee} \right> $ for all $1\leq k \leq i-2$, $ v_k-v_{k+1}=\left <\lambda + \rho_{\pi}, \alpha_k^{\vee} \right> $ for all $k\geq i+2$, and $\displaystyle \sum_{i=1}^{n+1}v_i=-\binom{n+1}{2}$.

Let $T(v)$ be the Gelfand--Tsetlin tableau with entries
\begin{align*}
v_{rs}=\begin{cases}
v_1+i-1, & \ \ \text{if}\ \  r=s=1,\\
v_1+i+1-r, & \ \ \text{if}\ \  1<r\leq i  \text{ and } s=r-1,\\
v_1, & \ \ \text{if}\ \  r>i   \text{ and } s=i,\\
v_2+i+1-r, & \ \ \text{if}\ \  1<r\leq i  \text{ and } s=r,\\
v_2, & \ \ \text{if}\ \  r>i  \text{ and } s=r,\\
v_3, & \ \ \text{if}\ \  r\geq 3  \text{ and } s=1,\\
v_{s+2}, & \ \ \text{if}\ \  r\geq 4  \text{ and } 2\leq s <i,\\
v_s, & \ \ \text{if}\ \ r>s\geq i+1.
\end{cases}
\end{align*}

Consider the set of relations $\mathcal{C}= \mathcal{C}^+ \cup \mathcal{C}^-$, where 
  $$\mathcal{C}^+ := \{((r+1,s);(r,s))\ |\ 1\leq s< r\leq n \text{ or } r=s=1\} \bigcup \{((r+1,r+1);(r,r))\ |\ i+1\leq  r\leq n\} $$
      $$\mathcal{C}^- :=\{ ((r,r);(r+1,r+1))\ |\ 1 < r\leq i \} \bigcup \{((r,s);(r+1,s+1))\ |\ 1\leq s< r\leq n\} .$$

In this case, $\mathcal{C}$ is an admissible set of relations, $T(V)$ is a $\mathcal{C}$-realization, and $V_{\mathcal{C}}(T(v))$ is a simple module by Theorem \ref{thm-irr}. Moreover,  $V_{\mathcal{C}}(T(v))=U(\mathfrak{g})T(v)$, $E_{12}(T(v))=E_{i+1,1}(T(v))=E_{2,i+2}(T(v))=E_{k,k+1}(T(v))=0$ for all $k\notin \{i-1,i,i+1\}$, $H_1 ( T(v))=\left<\lambda,\alpha_i^{\vee} \right> T(v)$, $ \displaystyle \sum_{k=1}^i (-H_k)  (T(v))=\left<\lambda,\alpha_{i-1}^{\vee} \right> T(v)$, $\displaystyle \sum_{k=2}^{i+1} H_k(T(v))=\left<\lambda,\alpha_{i+1}^{\vee} \right> T(v)$,  $H_{j+2}(T(v))=\left<\lambda,\alpha_j^{\vee} \right> T(v)$ for all $1\leq j \leq i-2$  and $H_j (T(v))=\left<\lambda,\alpha_j^{\vee} \right> T(v)$ for all $j \geq i+2$. Hence, $V_{\mathcal{C}}(T(v)) \simeq L_i$.

\end{proof}

\

\section{Induced relation modules}\label{Section: Induced relation modules}
In this section we fix $f=E_{21}$ and $\Gamma=\Gamma_{st}$.

\subsection{Localization of relation modules}

\begin{Lem}\label{inj-surj-E21}   Suppose that $T(v)$ is a $\mathcal{C}$-realization for some admissible  set of relations $\mathcal{C}$.
\begin{enumerate}[a)]
\item  $f$ is injective on  $V_{\mathcal{C}}(T(v))$ if and only if $((1,1);(2,1)) \notin \mathcal{C}$ and $((1,1);(2,2)) \notin \mathcal{C}$. 
\item $f$ is surjective on  $V_{\mathcal{C}}(T(v))$ if and only if $((2,1);(1,1)) \notin \mathcal{C}$ and $((2,2);(1,1)) \notin \mathcal{C}$.
\end{enumerate}

\end{Lem}

\begin{proof} 
For any $T(w)\in \mathcal B_{\mathcal{C}}(T(v))$ we have  $f(T(w))=T(w-\delta^{11})\in \mathcal B_{\mathcal{C}}(T(v))$, since $((1,1);(2,1))\notin \mathcal{C}$ and $((1,1);(2,2)) \notin \mathcal{C}$.  Suppose that $0\not= u=\displaystyle \sum_{i\in I} a_i T(w^i) \in V_{\mathcal{C}}(T(v))$. Then $$\displaystyle f(u)=f\bigg( \sum_{i\in I} a_i T(w^i) \bigg)=\sum_{i \in I} a_i T(w^i-\delta^{11})  \not=0,$$ as $a_i\not=0$, $T(w^i-\delta^{11})\in V_{\mathcal{C}}(T(v))$ and $T(w^i-\delta^{11})\not = T(w^j-\delta^{11})$ for all $i\not=j \in I$. 
 On the other hand, if $((1,1);(2,1))\in \mathcal{C}$ or $((1,1);(2,2)) \in \mathcal{C}$ then there exists $T(w)\in \mathcal B_{\mathcal{C}}(T(v))$ such that $w_{11}-w_{21}=1$ or $w_{11}-w_{22}=1$, and hence $f(T(w))=T(w-\delta^{11})=0$, as $T(w-\delta^{11})\notin \mathcal B_{\mathcal{C}}(T(v))$. This shows item (a). 

If $((2,1);(1,1)) \notin \mathcal{C}$ and $((2,2);(1,1)) \notin \mathcal{C}$ then for any $T(w)\in \mathcal B_{\mathcal{C}}(T(v))$, we have  $T(w+\delta^{11})\in \mathcal B_{\mathcal{C}}(T(v))$ and $f(T(w+\delta^{11}))=T(w)$. Hence, for any $u\in V_{\mathcal{C}}(T(v))$,  $$u=\sum_{i\in I} a_i T(w^i)=\sum_{i\in I}  a_i f(T(w^i+\delta^{11}))=f\left(\sum_{i\in I}  a_i T(w^i+\delta^{11})\right).$$
On the other hand, assume that $((2,1);(1,1)) \in \mathcal{C}$ or $((2,2);(1,1)) \in \mathcal{C}$ and $f$ is surjective on $V_{\mathcal{C}}(T(v))$. Then there exists $T(w)\in  \mathcal B_{\mathcal{C}}(T(v))$ with  $w_{11}=w_{21}$ or $w_{11}=w_{22}$ and $u\in V_{\mathcal{C}}(T(v))$, such that $$f(u)=\sum_{i\in I} a_i T(w^i-\delta^{11})=T(w).$$
Hence, $T(w)=T(w^i-\delta^{11})$ for some $i\in I$. As $T(w^i)\in \mathcal B_{\mathcal{C}}(T(v))$, we have  $0=w_{21}-w_{11}=w^i_{21}-w^i_{11}+1\in \mathbb{Z}_{>0}$ or $0=w_{22}-w_{11}=w^i_{22}-w^i_{11}+1\in \mathbb{Z}_{>0}$. This implies (b).
\end{proof}

Next we consider the twisted localization of  $V_{\mathcal{C}}(T(v))$ with respect to $f$ and its tableaux realization.

\begin{Lem}\label{localiz-E21}  Let $\mathcal{C} $ be an admissible set of relations such that $$\mathcal{C} \cap \{((1,1);(2,1)),  ((1,1);(2,2)) \}= \emptyset$$ and  $T(v)$ a $\mathcal{C}$-realization. If
 $\mathcal{D} =  \mathcal{C} \setminus \{((2,1);(1,1)),((2,2);(1,1))\}$ then the localized module  $D_f(V_{\mathcal{C}}(T(v)))$  is isomorphic to $V_{\mathcal{D}}(T(v))$.
\end{Lem}

\begin{proof} Set $M=V_{\mathcal{C}}(T(v))$ and $N=V_{\mathcal{D}}(T(v))$. 
By Lemma \ref{inj-surj-E21}, the action of $f$  on $M$ is injective but not  bijective. Hence $M$ is a proper submodule of $D_f M$. 
On the other hand, consider the set of relations $\mathcal{C}_1=\mathcal{D}\cup \{((1,1);(2,1)) \} $. By Theorem \ref{thm-irr}, $V_{\mathcal{C}_1}(T(v+\delta^{11}))$ is a simple module isomorphic to $N/M$ and hence, $M$ is a maximal submodule of $N$.  Finally, since $f$ is bijective on $N$, by Proposition \ref{loc-max} we have $N\simeq D_f M$. 
\end{proof}

Under the assumptions of Lemma \ref{localiz-E21} the following is straightforward.

\begin{Lem}\label{lem-formulas} For any tableau $T(w)\in \mathcal{B}_{\mathcal{C}}(T(v))$ and any $x\in \mathbb{C}$ denote by $T(w)^x$ the image of $T(w)$ as an element of  the twisted module $D_f^{x}(V_{\mathcal{C}}(T(v)))$. Then we have
 \\

   \begin{tabular}{rl}
       $E_{12}\cdot T(w)^x$  =& $-(w_{11}+x-w_{21})(w_{11}+x-w_{22}) T(w+\delta^{11})^x$.
    \end{tabular}  

   \begin{tabular}{rl}
   $E_{k,k+1}\cdot T(w)^x$ = &  $\displaystyle -\sum_{i=1}^{k}\left(\dfrac{\prod_{j=1}^{k+1}(w_{ki}-w_{k+1,j})}{\prod_{j\neq i}^{k}(w_{ki}-w_{kj})}\right)T(w+\delta^{ki})^x$, for all $k=2, \dots, n$.
   \end{tabular}
  
  \begin{tabular}{rl}
  $E_{21}\cdot T(w)^x$ = & $T(w-\delta^{11})^x$.
   \end{tabular}

\begin{tabular}{rl}
       $E_{32}\cdot T(w)^x=$  &  $\displaystyle \frac{  w_{21}-(w_{11}+x)}{w_{21}-w_{22}}T(w-\delta^{21})^x + \frac{ w_{22}-(w_{11}+x)}{w_{22}-w_{21}}T(w-\delta^{22})^x$.\end{tabular} 
   
   \begin{tabular}{rl}
   $E_{k+1,k}\cdot T(w)^x$= & $ \displaystyle \sum_{i=1}^{k}\left(\frac{\prod_{j=1}^{k-1}(w_{ki}-w_{k-1,j})}{\prod_{j\neq i}^{k}(w_{ki}-w_{kj})}\right)T(w-\delta^{ki})^x$, for all $k=3, \dots, n$.
    \end{tabular}
    
     \begin{tabular}{rl}
 $H_{1}\cdot T(w)^x$ = &  $(2(w_{11}+x) -(w_{21}+w_{22})-1) T(w)^x$.
  \end{tabular}

  \begin{tabular}{rl}
  $H_{2}\cdot T(w)^x$ = & $\left(2(w_{21}+w_{22})-(w_{31}+w_{32}+w_{33})-(w_{11}+x)-1\right)T(w)^x$.
    \end{tabular}
    
      \begin{tabular}{rl}
  $H_{k} \cdot T(w)^x$ = &  $\left(2\sum\limits_{i=1}^{k}w_{ki}-\sum\limits_{i=1}^{k+1}w_{k+1,i}-\sum\limits_{i=1}^{k-1}w_{k-1,i}-1\right)T(w)^x$, for all $k=3, \dots, n$.\\ 
  &
  \end{tabular}  

\end{Lem}

\begin{Th}\label{twistedE21} Let $M$ be a $\Gamma$-relation module with an injective action of $f$. Then  $D_f^{x} M$ is a $\Gamma$-relation Gelfand--Tsetlin module.
\end{Th}

\begin{proof} By hypothesis, $M \simeq V_{\mathcal{C}}(T(v))$ for some 
$\mathcal{C}$-realization $T(v)$. By  Lemma \ref{localiz-E21}, 
the set $\{ T(v+\ell\delta^{11})\ |\ \ell\in\mathbb{Z}\}$ is a basis of the localized module $D_f M$.   
Define the linear isomorphism $\phi$ from $D^x_f(V_{\mathcal{C}}(T(v)))$ to $V_{\mathcal{D}}(T(v+x\delta^{11}))$, such that    $\phi(T(w)):=T(w+x\delta^{11})= T(w)^x\in \mathcal{B}_{\mathcal{D}}(T(v+x\delta^{11}))$ for any $T(w)\in \mathcal{B}_{\mathcal{C}}(T(v))$. 
 Comparing the twisted action of $\mathfrak{g}$ from Lemma  \ref{lem-formulas} with the Gelfand--Tsetlin formulas, we have
$$\phi(g\cdot T(w)^x)=g\phi(T(w)^x)$$ for any $g\in  \mathfrak{g}$.
Hence, $\phi$ is an isomorphism of modules, which completes the proof. 

\end{proof}

\begin{Co}\label{cor-twisted-isom}
Let $M$ be a simple $\Gamma$-relation Gelfand--Tsetlin module. 
\begin{enumerate}[a)]
    \item If $f$ is bijective on $M$, then $M$ is isomorphic to   $D_f^{x}(N)$ for some simple $\Gamma$-relation Gelfand--Tsetlin module $N$ with an injective action of  $f$  and $x\in \mathbb{C}\setminus \mathbb{Z}$.
    \item If $f$ is surjective on $M$ but not injective, then $M$ is isomorphic to   $D_f(N)/N$ for some simple $\Gamma$-relation Gelfand--Tsetlin module $N$ with an injective action of  $f$.
\end{enumerate}
\end{Co}

\begin{proof}
As $M$ is a simple relation module, then $M \simeq V_{\mathcal{C}}(T(v))$ for some Gelfand--Tsetlin tableau and  a maximal admissible set of relations $\mathcal{C}$ satisfied by $T(v)$.

Part (a): Suppose that $f$ is bijective on $M$. By Lemma \ref{inj-surj-E21}, we have  $ v_{11}-v_{21}\notin \mathbb{Z}$ and $ v_{11}-v_{22}\notin \mathbb{Z}$. Without loss of generality we assume that $ v_{21}-v_{22}\notin \mathbb{Z}$.
 In this case, let $x=v_{11}-v_{21}$ and consider the Gelfand--Tsetlin tableau $T(v')=T(v+x\delta^{11})$. Then  $\mathcal{C'}=\mathcal{C} \cup \{ ((2,1);(1,1))  \}$ is the maximal admissible set of relations satisfied by $T(v')$. Hence, by Theorem \ref{thm-irr}, $N=V_{\mathcal{C'}}(T(v'))$ is a simple $\Gamma$-relation Gelfand--Tsetlin module. By Lemma \ref{inj-surj-E21}, $f$ is injective on $N$ but not surjective.  Then $D_f^{x} N\simeq M$ by Theorem \ref{twistedE21}.

Part (b): Suppose that $f$ is surjective on $M$, but not injective. By Lemma \ref{inj-surj-E21}, we can assume that $ v_{11}-v_{21}\in \mathbb{Z}_{>0}$ and $ v_{11}-v_{22}\notin \mathbb{Z}$. Set again $x=v_{11}-v_{21}$ and  $T(v')=T(v+x\delta^{11})$. Then  $\mathcal{C'}=\bigg(\mathcal{C}\setminus{\{ ((1,1);(2,1))  \}} \bigg) \cup \{ ((2,1);(1,1))  \}$ is the maximal admissible set of  relations satisfied by $T(v')$. Hence, $N=V_{\mathcal{C'}}(T(v'))$ is a simple relation Gelfand--Tsetlin module by Theorem \ref{thm-irr}.  From Lemma \ref{inj-surj-E21} we have that $f$ is injective on $N$, and that the localized module $D_f(N)$  is isomorphic to $V_{\mathcal{C}\setminus{\{ ((1,1);(2,1))  \}}}(T(v'))$ (cf. Lemma \ref{localiz-E21}). Moreover, $N$ is a maximal submodule of  $D_f(N)$, and hence $D_f(N)/N\simeq V_{\mathcal{C}}(T(v'))$ is a simple module. We conclude that $D_f(N)/N\simeq V_{\mathcal{C}}(T(v))$.

\end{proof}

\subsection{ \texorpdfstring{$\mathfrak{sl}_2$}{sl2}-induced relation modules}
Let ${\mathfrak p}\subset \mathfrak{g}$ be a parabolic subalgebra of $\mathfrak{g}$ with the Levi factor isomorphic to $\mathfrak a_{\alpha}=\mathfrak{sl}_2+\mathfrak{h}$ based on the root $\alpha=\alpha_1\in \pi$.
Let $V =V(\gamma,\mu)$ be a simple cuspidal weight $\mathfrak a_{\alpha}$-module, where $\mu\in \mathfrak{h}^*$  is such that $V_{\mu}\neq 0$, and $\gamma\in \mathbb C$ is the eigenvalue of the Casimir element 
$c_{\alpha} = (E_{11}-E_{22}+1)^2+4E_{21}E_{12}$
 of 
$\mathfrak a_{\alpha}$. Denote by $L(\gamma,\mu)=L_\mathfrak{p}^{\mathfrak{g}}(\{ \alpha \}, V)$ the unique simple quotient of the induced module $\mathop {\rm Ind}_{{\mathfrak p}}^{\mathfrak{g}}(\{ \alpha \},V)$ with the trivial action of the radical of ${\mathfrak p}$ on $V$. 
Let   $\Gamma_{\alpha}$ be any Gelfand--Tsetlin subalgebra corresponding to the flag containing $\mathfrak a_{\alpha}$ and $\mathfrak h$.  

We have

\begin{Th}\label{sl2 Induced} Let $n>1$.
The module $L(\gamma,\mu)$ is a $\Gamma_{\alpha}$-relation Gelfand--Tsetlin $\mathfrak{g}$-module if and only if $L(\lambda)$ is a $\Gamma_{\alpha}$-relation highest weight $\mathfrak{g}$-module, where $\left<\mu-\lambda,\alpha_1^{\vee} \right>=2x$, $\left<\mu-\lambda,\alpha_2^{\vee} \right>=-x$, $\left<\mu-\lambda,\alpha_i^{\vee} \right>=0$ for each $i=3,\dots, n$, $\left<\lambda+\rho_{\pi},\alpha_1^{\vee} \right> \notin \mathbb{Z}_{\geq 0}$ and
 $(2x-\mu_1-1)^2=\gamma$. In this case $L(\gamma,\mu) \simeq D_f^{x}(L(\lambda))$ and $x$ satisfies the condition $x-\left<\mu+\rho_{\pi},\alpha_1^{\vee} \right> \notin\mathbb{Z}$. Moreover,  
\begin{enumerate}[a)]
    \item $L(\gamma,\mu)$ is bounded if and only if $L(\lambda)$ is bounded;
\item   If $\lambda$ is dominant then $s_1 \cdot \lambda$ is dominant and $\gamma\neq m^2$ for all $m \in \mathbb{Z}$. In this case, $Ann_{U(\mathfrak{g})} L(\gamma,\mu)=Ann_{U(\mathfrak{g})} L(\lambda)=Ann_{U(\mathfrak{g})} L(s_1 \cdot \lambda);$
\item If $\gamma=m^2$ for some $m \in \mathbb{Z}\setminus\{0\}$ then $L(\lambda)$ is bounded and  the weight $w \cdot \lambda$ is integral, for all $w \in W$.  In this case, $ Ann_{U(\mathfrak{g})} L( \lambda) \subset Ann_{U(\mathfrak{g})} L(\gamma,\mu) \subset Ann_{U(\mathfrak{g})} L(w \cdot \lambda)$, if $w \cdot \lambda$ is dominant.
\end{enumerate}
\end{Th}

\begin{proof}
Since $V$ is a simple dense $\mathfrak{sl}_2$-module, then $\gamma \neq (\mu_1-2k+1)^2$ for all $k\in \mathbb{Z}$, and hence $x \notin \mathbb{Z}$. On the other hand,$\gamma=m^2$ for some  $m \in \mathbb{Z}$ if and only if  $\left<\lambda+\rho_{\pi},\alpha_1^{\vee} \right> \in \mathbb{Z}$.  Consider the following two cases:  

\textbf{Case 1:} Let $\gamma \neq m^2$ for all $m\in \mathbb{Z}$. 
Assume that $L(\gamma,\mu)$ is a $\Gamma_{\alpha}$-relation Gelfand--Tsetlin $\mathfrak{g}$-module. Then  
$L(\gamma,\mu) \simeq V_{\mathcal{C}}(T(v))$ for some set of relations $\mathcal{C}$ and 
 a tableau  $T(v)=T(v_{ij})$, such that 
 \begin{align*}
v_{ij}=\begin{cases}
v_1+x, & \ \ \text{if}\ \  i=1\\
v_1, & \ \ \text{if}\ \  i>j=1 \\
v_j, & \ \ \text{if}\ \  j\geq 2,
\end{cases}
\end{align*}
with $v_1-v_2= \left<\mu+\rho_{\pi},\alpha_1^{\vee} \right>-2x $, $v_2-v_3= \left<\mu+\rho_{\pi},\alpha_2^{\vee} \right> +x $,  $v_j-v_{j+1}= \left<\mu+\rho_{\pi},\alpha_j^{\vee} \right> $ for each $3\leq j\leq n$ and $\displaystyle \sum_{j=1}^{n+1}v_j=-n-1$. On the other hand, $f$ is bijective on $V_{\mathcal{C}}(T(v))$ if and only if  $x-\left<\mu,\alpha_1^{\vee} \right> \notin\mathbb{Z}$, in which case  $$ \mathcal{C} \cap \{((1,1);(2,1)),   ((1,1);(2,2)), ((2,1);(1,1)), ((2,2);(1,1)) \}= \emptyset.$$ 
 Consider the  tableau $T(v')=T(v'_{ij})$ with  entries $v'_{11}=v_{21}$ and $v'_{ij}=v_{ij}$ for $i\neq 1$. Then $\mathcal{D}=\mathcal{C} \cup \{ ((2,1);(1,1))  \}$ is the maximal set of  relations satisfied by $T(v')$, since $\left<\lambda+\rho_{\pi},\alpha_1^{\vee} \right> \notin \mathbb{Z}$. Hence the module  $V_{\mathcal{D}}(T(v'))$ is simple by Theorem \ref{thm-irr}.  Then $V_{\mathcal{D}}(T(v'))\simeq U(\mathfrak{g})T(v')$,  $E_{k,k+1}(T(v'))=0$  and  $H_k( T(v))=\lambda_k T(v)$  for all $k=1, \dots, n$, $\lambda_k = \left<\lambda,\alpha_k^{\vee} \right>$. We get  $V_{\mathcal{D}}(T(v')) \simeq L(\lambda)$. Moreover,   $ c_{\alpha} (T(v'))=(\lambda_1+1)^2 T(v')=\gamma T(v')$. 
Applying Theorem \ref{twistedE21}, we conclude that $D_f^{x} L(\lambda)\simeq L(\gamma,\mu)$.

Conversely, let $L(\lambda)$ be  a $\Gamma_{\alpha}$-relation Gelfand--Tsetlin module. Then $L(\lambda) \simeq V_{\mathcal{C}}(T(v))$, where  $T(v)=T(v_{ij})$ is the Gelfand--Tsetlin tableau such that $v_{ij}=v_j$  with  $v_j-v_{j+1}= \left<\lambda+\rho_{\pi},\alpha_j^{\vee} \right> $ for each $1\leq j\leq n$, and $\displaystyle \sum_{j=1}^{n+1}v_j=-n-1$. Note that $\mathcal{C}$ is the maximal set of relations satisfied by $T(v)$. Without loss of generality we assume that $((2,1);(1,1))\in \mathcal{C}$ and $((1,1);(2,1))\notin \mathcal{C}$. Hence, the localized module $D_f^{x}(L(\lambda))$  is  isomorphic   to  $V_{\mathcal{D}}(T(v+x\delta^{11}))$, where $\mathcal{D}=\mathcal{C} \setminus{ \{ ((2,1);(1,1)) \}}$ by Theorem \ref{twistedE21}.
Given that   $x+v_{11}-v_{22} \notin \mathbb{Z}$, then $V_{\mathcal{D}}(T(v+x\delta^{11}))$ is a simple module. Further, the $\mathfrak{sl}_2$-module   $\mbox{span}_\mathbb{C} \{ T(v+(x +\ell)\delta^{11})\ | \ \ell \in \mathbb{Z} \}$ is isomorphic to $V$. By 
 \eqref{Gelfand--Tsetlin formulas}, we have  $E_{k,k+1}(T(v +(x+\ell)\delta^{11}))=0$ for all $k=2,\dots, n$ and $\ell \in \mathbb{Z}$. Hence, we get 
 an epimorphism of $U(\mathfrak{g})$-modules
$$\phi: M_\mathfrak{p}^{\mathfrak{g}}(\{ \alpha \}, V) \rightarrow  V_{\mathcal{D}}(T(v+x \delta^{11})),$$
 such that $\phi(u\otimes T(v))=u T(v)$ for all $u\in U(\mathfrak{g})$. Therefore,   $L(\gamma,\mu) \simeq V_{\mathcal{D}}(T(v+x \delta^{11}))$.
  
 \textbf{Case 2:} Let $\gamma=m^2$ for some $m \in \mathbb{Z}$. 
In this case $\left<\lambda+\rho_{\pi},\alpha_1^{\vee} \right> \in \mathbb{Z}_{< 0} $ and the construction is similar. We leave the details out.

Note that if  $\left<\lambda+\rho_{\pi},\alpha_1^{\vee} \right> =0$ then  $L(\gamma,\mu)$ is not a relation module. 
 The statement  a) is clear from the construction, while b)  and c) follow from  Proposition \ref{max-ann}, Corollary \ref{dom-ann} and Corollary \ref{Localization annihilator}.
 \end{proof} 

\

\subsection{Family of induced relation modules}

In this section we give an explicit construction  of a family of parabolically induced bounded $\Gamma$-relation modules. 

Fix complex  $\{u_{i}\}_{i=1,\ldots,n+1}$ and $\{v_{i}\}_{i=1,\ldots,n}$  satisfying conditions: 
 \begin{enumerate}[a)]
  \item  $u_{i}-v_1\notin\mathbb{Z}$ for any $1\leq i \leq n$.
 \item $v_{j}-v_{j+1}\in \mathbb{Z}_{>0}$ for any $1\leq j <n.$ 
\end{enumerate}

Let $T(v)$ be a Gelfand--Tsetlin tableau with entries
\begin{align}\label{tableufixed}
v_{ij}=\begin{cases}
u_i, & \ \ \text{if}\ \  j=1\\
v_{j-1}, & \ \ \text{if}\ \  j\neq 1
\end{cases}
\end{align}
for $1\leq j \leq i \leq n+1.$

Consider the set of relations $\mathcal{Q}= \mathcal{Q}^+ \cup \mathcal{Q}^-$, where 
\begin{center}
\begin{tabular}{rl}
        $\mathcal{Q}^+ :=$  & $\{((i+1,j);(i,j))\ |\ 2\leq j\leq i\leq n\}$ \\
      $\mathcal{Q}^- :=$  & $\{((i,j);(i+1,j+1))\ |\ 2\leq j\leq i\leq n\} $. \\
\end{tabular} 
\end{center}

 \begin{Lem}\label{Th: cuspidalBonded}
Let $T(v)$ be a Gelfand--Tsetlin tableau \eqref{tableufixed}. Then  $V_{\mathcal{Q}}(T(v))$  is a bounded dense $\mathfrak{g}$-module. Moreover, $V_{\mathcal{Q}}(T(v))$ is simple if and only if  $u_{i}-u_{i+1}\notin\mathbb{Z}$ for all $1\leq i \leq n$.
\end{Lem}
\begin{proof} Follows from  \cite[Lemmas 3.1, 3.2, 3.3]{Maz03}.
\end{proof}

If $\mathcal{C}$ is any admissible set of relations containing $\mathcal{Q}$ and $T(v)$ is a $\mathcal{C}$-realization, then   $V_{\mathcal{C}}(T(v))$ is a submodule of $V_{\mathcal{Q}}(T(v))$ and hence a bounded module. In particular, $V_{\mathcal{C}}(T(v))$ has finite length \cite{Fer90}.

For $m\in \{2 \dots, n\}$  consider the tableau $T^m(v)$ as in \eqref{tableufixed} satisfying the conditions
\begin{equation*}
\label{Rem: tableaux for slm induzed}
 u_i=u_m \,\, \text{for} \,\, i=m+1,\dots, n+1, \,\, \text{and} \,\,  u_i-u_{i+1}\notin\mathbb{Z} \,\, \text{for} \,\,i=1,2,\dots, m-1. 
 \end{equation*}
 
 Let $\mathcal{C}^m =\mathcal{Q} \cup \{((i+1,1);(i,1))\ |\ m\leq i\leq n\}$. Then  $T^m(v)$ is a $\mathcal{C}^m$-realization and the relation module $V_{\mathcal{C}^m}(T^m(v))$ is simple. 

Consider the subset $\Sigma_m=\{\alpha_1, \ldots, \alpha_{m-1}\}\subset \pi$  of simple roots  and the 
corresponding parabolic subalgebra
$\mathfrak p_m=\mathfrak p_{\Sigma_m}\subset \mathfrak g$. Then  $\mathfrak p_m$ has the Levi subalgebra  isomorphic to $\mathfrak{sl}_m+\mathfrak{h}$.
We have

\

\begin{Th}\label{Lem: Induced by sl(m)} For $m\in \{2, \dots, n\}$,  the  module $V_{\mathcal{C}^m}(T^m(v))$ is isomorphic to $L_{\mathfrak p_m}^{\mathfrak{g}}(\Sigma_m, V)$  for some simple  cuspidal $\mathfrak{sl}_m$-module $V$. 
\end{Th}

\begin{proof}
By construction,  $\mathcal{C}^m$ is the maximal set of admissible relations  satisfied by $T^m(v)$. Hence, $V_{\mathcal{C}^m}(T^m(v))$ is a simple Gelfand--Tsetlin $\mathfrak{g}$-module by Theorem \ref{thm-irr}. Let $\mathcal{D}^m=\mathcal{D}^+ \cup \mathcal{D}^-$, where  
\begin{align*}
\mathcal{D}^+ &= \{((i+1,j);(i,j))\ |\ 2\leq j\leq i\leq m-1\}\\
\mathcal{D}^- &=\{((i,j);(i+1,j+1))\ |\ 2\leq j\leq i\leq m-1\}.
\end{align*}
  Then $V=V_{\mathcal{D}^m}(T^m(v)) = \text{span}_\mathbb{C} {\mathcal B}_{\mathcal{D}^m}(T^m(v))$ is a simple cuspidal $\mathfrak{sl}_m$-module by Lemma \ref{Th: cuspidalBonded}, where ${\mathcal B}_{\mathcal{D}^m}(T^m(v))$  denotes the set of $\mathcal{D}^m$-realizations of the form $T^m(v+z)$, $z\in {\mathbb Z}_0^{\frac{m(m+1)}{2}}$.  It follows from \eqref{Gelfand--Tsetlin formulas} that $V$ is a $\mathfrak{p}_m$-module with the trivial action of $\mathfrak{u}^{+}_{\Sigma_m}$. Hence, we have a homomorphism 
$$\phi: M_{\mathfrak p}^{\mathfrak{g}}(\Sigma, V) \rightarrow  V_{\mathcal{C}^m}(T^m(v))$$
of $U(\mathfrak{g})$-modules, such that $\phi(u\otimes T^m(v))=u T^m(v)$ for all $u\in U(\mathfrak{g})$. Since $V_{\mathcal{C}^m}(T^m(v))$ is a simple $\mathfrak{g}$-module, then $\phi$ is surjective, and $L_{\mathfrak p}^{\mathfrak{g}}(\Sigma, V) \simeq V_{\mathcal{C}^m}(T^m(v))$. 
\end{proof}

\subsection{Localization of  highest weight modules with respect to \texorpdfstring{$E_{m1}$}{Em1} }
\label{subsection:Localization sln}

For any $2 <  m \leq n+1$ and any $k\leq m-1$ fix $i_{k}\in\{1,\ldots,k\}$.  Associated with the set $\{i_{1},\ldots,i_{m-1}\}$  define
$$
\varepsilon(i_1,\dots , i_{m-1}):=-\delta^{1i_1}-\delta^{2i_2}-\delta^{3i_3}-\ldots-\delta^{m-1,i_{m-1}} \in \mathbb{Z}_{0}^{\frac{(n+1)(n+2)}{2}}.
$$

Suppose that $T(v)$ is a $\mathcal{C}$-realization for some admissible  set of relations $\mathcal{C}$. For each $T(w)\in \mathcal B_{\mathcal{C}}(T(v))$ and any $1\leq i_k \leq k \leq m-1$  define 
\begin{equation*}
a(w,i_1,\dots , i_{m-1}):=\left\{
\begin{array}{cc}
0,& \text{ if } T(w+\varepsilon(i_1,\dots , i_{m-1}))\notin  \mathcal {B}{_\mathcal{C}}(T(v))\\
\displaystyle \prod_{s=2}^{m-1}\frac{\displaystyle \prod_{t\neq i_{s-1}}^{s-1}(w_{si_s}-w_{s-1,t})}{\displaystyle \prod_{t\neq i_s}^{s}(w_{si_s}-w_{st})},& \text{ if } T(w+\varepsilon(i_1,\dots , i_{m-1}))\in  \mathcal {B}{_\mathcal{C}}(T(v)).
\end{array}
\right.
\end{equation*}

One can easily check by direct computation the following analog of   \cite[Proposition 3.13]{FGR16}.

\begin{Pro}\label{defEm1} Let $\mathcal{C}$ be an admissible set of relations and $T(v)$ a $\mathcal{C}$-realization. If $T(w)\in \mathcal B_{\mathcal{C}}(T(v))$ then 

\begin{equation}\label{actEm1}
E_{m1}(T(w)) \ =  \displaystyle \mathop{\sum_{k=1,\dots, m-1}}_{(i_1,\ldots,i_{m-1})\in \{
1,\ldots,k \}^{m-1}} a(w,i_1,\dots , i_{m-1})\  T(w+\varepsilon(i_1,\dots , i_{m-1})),
\end{equation}
for $m\in \{3, \dots, n+1\}$.
\end{Pro}

\

\begin{Th}\label{inj-surj-Em1}  Let $T(v)$ be the Gelfand--Tsetlin tableau satisfying \eqref{tableufixed}, $\mathcal{C}$ an admissible set of relations containing $\mathcal{Q}$, for which $T(v)$ is a $\mathcal{C}$-realization. Then 
\begin{enumerate}[a)]
\item
$E_{m1}$ is injective on $V_{\mathcal{C}}(T(v))$ if and only if  $((m-1,1);(m,1)) \notin \mathcal{C};$
\item $E_{m1}$ is surjective on $V_{\mathcal{C}}(T(v))$  if and only if  $((m,1);(m-1,1)) \notin \mathcal{C}.$
\end{enumerate}
\end{Th}

\begin{proof} 
\begin{enumerate}[a)]
    \item For any $T(w)\in \mathcal B_{\mathcal{C}}(T(v))$ we have 
     $$T(w+\varepsilon(1,\dots , 1))=T(w-\delta^{11}-\delta^{21}-\delta^{31}-\ldots-\delta^{m-1,1})\in \mathcal B_{\mathcal{C}}(T(v)).$$
In fact, $(w_{i+1,1}-1)-(w_{i1}-1)=w_{i+1,1}-w_{i1}$ for all $i=1,\dots ,m-2$, $w_{m1}-(w_{m-1,1}-1)=w_{m1}-w_{m-1,1}+1\notin \mathbb{Z}_{\leq 1}$, and $(w_{i1}-1)-w_{i+1,2}=w_{i1}-w_{i+1,2}-1\notin
\mathbb{Z}$ for all $i=1,\dots ,m-1$. 
On the other hand, for any $s \in \{3,\dots m-1\}$ given that $w_{s1}-w_{s-1,2} \notin \mathbb{Z}$ and $w_{s-1,2}-w_{s-1,t} \in \mathbb{Z}$ for all $t \in \{2,\dots s-1\}$, we have $w_{s1}-w_{s-1,t} \notin \mathbb{Z}$ for all $t \in \{2,\dots s-1\}$. Hence 
$a(w,1,\dots , 1)\neq 0.$
Using the fact that $T(w+\varepsilon(1,\dots , 1)) \neq  T(w+\varepsilon(i_1,\dots , i_{m-1}))$ for all 
$(i_1,\ldots,i_{m-1})\neq (1,\dots,1)$,  
we obtain (cf. formula \eqref{actEm1})

$$ E_{m1}(T(w))=a(w,1,\dots , 1)\  T(w+\varepsilon(1,\dots , 1))+$$ 
$$+\displaystyle \sum_{k=2}^{m-1}\sum_{(i_1,\ldots,i_{m-1})\neq (1,\dots,1)} a(w,i_1,\dots , i_{m-1})\  T(w+\varepsilon(i_1,\dots , i_{m-1}))\neq 0.$$

Now, suppose that $0\not= u=\displaystyle \sum_{i \in I} c_i T(w^i) \in V_{\mathcal{C}}(T(v))$ with $c_i\not=0$ for all $i\in I$. Then 
$$E_{m1}(u)  =  \displaystyle \mathop{\sum_{i\in I, \ k=1,\dots, m-1}}_{(i_1,\ldots,i_{m-1})\in \{
1,\ldots,k \}^{m-1}} c_i \ a(w^i,i_1,\dots , i_{m-1})\  T(w^i-\varepsilon(i_1,\dots , i_{m-1})) \neq 0.$$
In fact,  for any $i \in I$ we have $c_i a(w^i,1,\dots , 1) \ T(w^i+\varepsilon(1,\dots , 1)) \neq 0$ and $T(w^i+\varepsilon(1,\dots , 1)) \neq  T(w^i+\varepsilon(i_1,\dots , i_{m-1}))$ for all   $(i_1,\ldots,i_{m-1})\neq (1,\dots,1).$
Also, as $T(w^i) \neq T(w^j)$ we obtain $T(w^i+\varepsilon(1,\dots , 1)) \neq T(w^j+\varepsilon(1,\dots , 1))$ for all $i\neq j \in I$.
Finally, suppose that for any $i \in I$ there exists $j\in I$ such that $j\neq i$ and $T(w^i+\varepsilon(1,\dots , 1)) = T(w^j+\varepsilon(i_1,\dots , i_{m-1}))$ for some $(i_1,\ldots,i_{m-1})\neq (1,\dots,1).$ Then there exists $s\in \{2,\dots m-1\}$ with $w^i_{s1}-1=w^j_{s1}$, and hence $$\displaystyle \sum_{j\in I} w^j_{s1}=\sum_{i\in I} (w^i_{s1}-1)=\sum_{i\in I} w^i_{s1}-\#I=\sum_{j\in I} w^j_{s1}-\#I,$$ which is a contradiction. Thus, there  exists $i\in I$ such that $T(w^i+\varepsilon(1,\dots , 1)) \neq T(w^j+\varepsilon(i_1,\dots , i_{m-1}))$ for all $j\neq i$ and for all $(i_1,\ldots,i_{m-1})\neq (1,\dots,1).$

Conversely, let $((m-1,1);(m,1)) \in \mathcal{C}$. By the hypothesis, $$\{((m-1,j);(m,j+1))\ | \ j=2,\dots, m-1 \} \subset  \mathcal{C}$$ and  hence there exists  $T(w)\in \mathcal B_{\mathcal{C}}(T(v))$ such that $w_{m-1,1}-w_{m1}=w_{m-1,j}-w_{m,j+1}=1$  for each $j=2,\dots m-1$. Therefore $T(w+\varepsilon(i_1,\dots , i_{m-1}))\notin  \mathcal B_{\mathcal{C}}(T(v))$ for any $1\leq i_k \leq k \leq m-1$ and $E_{m1}(T(w))=0$. 

\item First, note that for any $T(w)\in \mathcal B_{\mathcal{C}}(T(v))$  we have $$T(w'):=T(w-\varepsilon(1,\dots , 1))\in \mathcal B_{\mathcal{C}}(T(v)),$$ 
since $((m,1);(m-1,1)) \notin \mathcal{C}$.

Using \eqref{actEm1} we obtain
\begin{center}
\begin{tabular}{rl}
        $E_{m1}(T(w'))=$  & $b(w,1,\dots , 1)\  T(w)$ \\
        + & $ \displaystyle \mathop{\sum_{k=2,\dots, m-1}}_{(i_2,\ldots,i_{m-1})\neq (1,\dots,1)} b(w,i_2,\dots , i_{m-1})\  T(w+\varepsilon'(i_2,\dots , i_{m-1})),$ 
\end{tabular} 
\end{center}
with $
\varepsilon'(i_2,\dots , i_{m-1}):=\varepsilon(i_1,\dots , i_{m-1})-\varepsilon(1,\dots , 1)
$, 
and \\

$b(w,i_2,\dots , i_{m-1}):=$
\begin{equation*}
= \left\{
\begin{array}{cc}
0,& \text{if } T(w+\varepsilon'(i_2,\dots , i_{m-1}))\notin  \mathcal {B}{_\mathcal{C}}(T(v))\\
\displaystyle \prod_{s=2}^{m-1}\frac{\displaystyle \prod_{t\neq i_{s-1}}^{s-1}(w_{si_s}-w_{s-1,t}+\delta_{1i_s}-\delta_{1t})}{\displaystyle \prod_{t\neq i_s}^{s}(w_{si_s}-w_{st}+\delta_{1i_s}-\delta_{1t})},& \text{if } T(w+\varepsilon'(i_2,\dots , i_{m-1}))\in  \mathcal {B}{_\mathcal{C}}(T(v)).
\end{array}
\right.
\end{equation*}

In particular, for any $s \in \{3,\dots m-1\}$ we have  $w_{s1}-w_{s-1,t} \notin \mathbb{Z}$ for any $t \in \{2,\dots s-1\}$, and hence 
$$b(w,1,\dots , 1)=\displaystyle \prod_{s=2}^{m-1}\frac{\displaystyle \prod_{t=2}^{s-1}(w_{s1}-w_{s-1,t})}{\displaystyle \prod_{t=2}^{s}(w_{s1}-w_{st}+1)} \neq 0.$$

As $T(w) \in\mathcal B_{\mathcal{C}}(T(v))$, we have  $w_{ij}-w_{i+1,j+1}\in \{1,2,\dots, d_{ij}\}$, where $d_{ij}:=v_{j-1}-v_{n+j-i}-n+i\in \mathbb{Z}_{>0}$. We consider the following cases:\\

\textbf{Case I:} $w_{22}-w_{33}\in \{1,2,\dots, d_{22}\}$   and $w_{ij}-w_{i+1,j+1}=1$ for all $3\leq  i\leq n$ and $2\leq j\leq i$.
In this case,  $T(w+\varepsilon'(i_2,\dots , i_{m-1}))\notin \mathcal B_{\mathcal{C}}(T(v))$ for all $(i_2,\dots , i_{m-1})$ such that $i_k>1$ for some $k=3,\dots, m-1$, since  $(w_{ij}-1)-w_{i+1,j+1}=0$ for all $3\leq i \leq n$ and $2\leq j\leq i$. Therefore:
$$E_{m1}(T(w'))=b(w,1,1,\dots , 1)  T(w) +b(w,2,1,\dots , 1)  T(w+\delta^{21}-\delta^{22})  $$
If $((1,1);(2,1)) \in \mathcal{C}$ or  $((3,1);(2,1)) \in \mathcal{C}$, then $T(w+\delta^{21}-\delta^{22}) \notin \mathcal B_{\mathcal{C}}(T(v))$, when $w_{11}-w_{21}=1$ or $w_{31}-w_{21}=0$.
 Hence $b(w,2,1,\dots , 1)=0$, and $E_{m1}(T(w'))=b(w,1,1,\dots , 1)  T(w)$.

Suppose now that $((1,1);(2,1)) \notin \mathcal{C}$ and $((3,1);(2,1)) \notin \mathcal{C}$ and consider the following cases:
 \begin{enumerate}[i)]
            \item $w_{22}-w_{33}=1$.
            
In this case  $T(w+\delta^{21}-\delta^{22})\notin \mathcal B_{\mathcal{C}}(T(v))$ implying $E_{m1}(T(w'))=b(w,1,1,\dots , 1)T(w)$.

 \item  $w_{22}-w_{33}=i \in \{2,3,4,\dots, d_{22}\}$.
            
            In this case $T(w+\delta^{21}-\delta^{22})\in \mathcal B_{\mathcal{C}}(T(v))$ as $(w_{22}-1)-w_{33}=w_{22}-w_{33}-1=i-1$ and  $w_{32}-(w_{22}-1)=w_{32}-w_{22}+1\in \mathbb{Z}_{>0}$. 
           By the induction on $i$ we have  
$T(w+\delta^{21}-\delta^{22})=E_{m1}(T(w_1)) $, where $T(w_1)$ is a tableau in $V_{\mathcal{C}}(T(v))$.
So, $E_{m1}(T(w'))=b(w,1,1,\dots , 1)  T(w) +b(w,2,1,\dots , 1)  T(w+\delta^{21}-\delta^{22})$.
Consequently, $T(w)=E_{m1}(T(w_0))$ for some $T(w_0)\in  V_{\mathcal{C}}(T(v))$.
 \end{enumerate}
\

 \textbf{Case II:}  $w_{22}-w_{33}\in \{1,2,\dots, d_{22}\}$, $w_{32}-w_{43}= j \in  \{2,3,4,\dots, d_{32}\}$  and $w_{ij}-w_{i+1,j+1}=w_{33}-w_{44}=1$ for all $4\leq  i\leq n$ and $2\leq j\leq i$.
 
In this case,  $T(w+\varepsilon'(i_2,\dots , i_{m-1}))\notin \mathcal B_{\mathcal{C}}(T(v))$ for all $(i_2,\dots , i_{m-1})$ such that $i_3=3$ or $i_k>1$ for any $k=4,\dots, m-1$, as $(w_{ij}-1)-w_{i+1,j+1}=(w_{33}-1)-w_{44}=0$ for all $4\leq i \leq n$ and $2\leq j\leq i$.
 
\begin{enumerate}[i)]
            \item If $w_{22}-w_{33}=1$ then $T(w+\delta^{21}-\delta^{22})$ and  $T(w+\delta^{21}-\delta^{22}+\delta^{31}-\delta^{32})$ do not belong $\mathcal B_{\mathcal{C}}(T(v))$. On the other hand, $T(w+\delta^{31}-\delta^{32})\in \mathcal B_{\mathcal{C}}(T(v))$, as $w_{42}-(w_{32}-1) \in \mathbb{Z}_{>0}$, $(w_{32}-1)-w_{43}=j-1\geq 2$ and $(w_{32}-1)-w_{22}=w_{43}-w_{44}+j-3 \in \mathbb{Z}_{> 0}$. Following the case I-(i), we conclude that there exists $T(w_1)\in V_{\mathcal{C}}(T(v))$ such that  $T(w+\delta^{31}-\delta^{32})=E_{m1}(T(w_1))$.
Therefore, $E_{m1}(T(w')))=b(w,1,1,\dots , 1)T(w)+b(w,1,2,\dots , 1)E_{m1}(T(w_1))$.

  \item Assume $w_{22}-w_{33}=i \in  \{2,3,\dots, d_{22}\}$.
   We immediately get that         
 $T(w+\delta^{21}-\delta^{22})$ and  $T(w+\delta^{21}-\delta^{22}+\delta^{31}-\delta^{32})$ belong to $\mathcal B_{\mathcal{C}}(T(v))$. 
 On the other hand, $T(w+\delta^{31}-\delta^{32})\in \mathcal B_{\mathcal{C}}(T(v))$ if and only if $w_{43}-w_{44} \geq i-j+2$, since $w_{42}-(w_{32}-1) \in \mathbb{Z}_{>0}$ and       $(w_{32}-1)-w_{22}=w_{43}-w_{44}+j-i-2 $. Then, following the cases  I and II-(i), there exists $T(w_k)  \in V_{\mathcal{C}}(T(v))$ for each $k=1,2,3$ such that $T(w+\delta^{21}-\delta^{22})=E_{m1}\left(T(w_1)\right)$, $T(w+\delta^{21}-\delta^{22}+\delta^{31}-\delta^{32})=E_{m1}\left(T(w_2)\right)$ and $T(w+\delta^{31}-\delta^{32})=E_{m1}\left(T(w_3)\right)$ (if $w_{43}-w_{44}< i-j+2$, then $T(w_3)=0$).

Hence, 
\begin{align*}
  E_{m1}( T(w'))= & b(w,1,1,\dots , 1)T(w)+b(w,2,1,\dots , 1)E_{m1}(T(w_1))\\
    & +b(w,2,2,\dots , 1)E_{m1}(T(w_2))+ b(w,1,2,\dots , 1)E_{m1}(T(w_3)).
\end{align*}

 \end{enumerate}

Repeating the process, after finitely many steps we obtain that for any $T(w)\in \mathcal B_{\mathcal{C}}(T(v))$ there exists  $T(w_0)\in  V_{\mathcal{C}}(T(v))$ such that $T(w)=E_{m1}(T(w_0))$ which implies the surjectivity of $E_{m1}$.

Conversely, assume that $E_{m1}$ is surjective on $V_{\mathcal{C}}(T(v))$ but $((m,1);(m-1,1)) \in \mathcal{C}$.
 Choose $T(w)\in  \mathcal B_{\mathcal{C}}(T(v))$ such that   $w_{m-1,1}=w_{m1}$. As $((m,j);(m-1,j)) \in \mathcal{C}$ for all $j=2,\dots, m-1$, we can assume without loss of generality  that $w_{m-1,j}=w_{m,j}$ for all $j=2,\dots, m-1$. On the other hand, there exists $u\in V_{\mathcal{C}}(T(v))$ such that
$$E_{m1}(u)  =  \displaystyle \mathop{\sum_{i\in I, \ k=1,\dots, m-1}}_{(i_1,\ldots,i_{m-1})\in \{
1,\ldots,k \}^{m-1}} c_i \ a(w^i,i_1,\dots , i_{m-1})\  T(w^i-\varepsilon(i_1,\dots , i_{m-1})) =T(w).$$
Hence $T(w)=T(w^i-\varepsilon(i_1,\dots , i_{m-1}))$ for some $i\in I$ and $1\leq i_k \leq k \leq m-1$. Then there exists  $j\in\{1,\dots, m-1\}$ such that  $0=w_{mj}-w_{m-1,j}=w^i_{mj}-w^i_{m-1,j}+1>0$. Therefore, $((m,1);(m-1,1)) \notin \mathcal{C}$.
\end{enumerate}
\end{proof}

\

\begin{Rem}
Note that in Theorem \ref{inj-surj-Em1}  the set of relations $\mathcal{C}$ does not need  to be neither indecomposable  nor maximal set  satisfied by $T(v)$. This is the case, for example, when $u_{i+1}-u_i\in \mathbb{Z}_{\geq 0}$ for any $1\leq i <n$ and $\mathcal{C}=\{((i+1,1);(i,1))\ |\ 1\leq i\leq n\}$ or $u_{i+1}-u_i\notin \mathbb{Z}$ for any $1\leq i <n$ and $\mathcal{C}=\emptyset$.
\end{Rem}

\

Let $F:=\{E_{m_i1}\ |\ i=1,\dots,k \}$ such that, $m_i \in \{2,\dots, n+1\}$ for each $i=1,\dots, k$. We have

\begin{Pro} \label{localiz-F}  Let  $\lambda\in \mathfrak{h}^*$.
\begin{enumerate}[a)]
\item If $\lambda$
 satisfies conditions b), c) or e) for $i=1$ of Corollary \ref{prop:relationbounded}, then $D_F L(\lambda)$ is a bounded relation Gelfand-Tsetlin module.
 \item If $M=L(\lambda)$ and $F_m=\{E_{m1} \}$ for some $m\in \{2,\dots, n+1\}$ such that $E_{m1}$ acts injectivity on $M$, then $D_{F_m}(M)/M$ is simple.
 \end{enumerate}
\end{Pro}

\begin{proof} Suppose that $\lambda$ satisfies the condition b) of Corollary \ref{prop:relationbounded} (the proof of other cases is similar). Then $L(\lambda) \simeq V_{\mathcal{C}}(T(v))$ where $T(v)$ is the Gelfand--Tsetlin tableau \eqref{tableufixed} such that  $u_i=u_1$ for all $i=2,3, \dots, n+1$ and 
$\mathcal{C}=\mathcal{Q} \cup \{((i+1,1);(i,1))\ |\ 1\leq i\leq n\}$. 
Set $M=V_{\mathcal{C}}(T(v))$ and $N=V_{\mathcal{D}}(T(v))$, where $\mathcal{D}= \mathcal{C} \setminus \{((m_i,1);(m_i-1,1))\  |\ i=1,\dots,k \}$.
Then $M\subset D_F M \subset N$ by Theorem \ref{inj-surj-Em1}. 

Suppose first that  $k=1$. Then  $F=F_m=\{E_{m1} \}$.  Since   $E_{m1}$ is injective but not bijective on $M$, we conclude by Theorem \ref{inj-surj-Em1} that $M$ is a proper submodule of $D_{F_m }M$. 
On the other hand, consider the set of relations $\mathcal{D}_m=\mathcal{D}\cup \{((m-1,1);(m,1)) \} $. Then $V_{\mathcal{D}_m}(T(v+\delta^{m-1,1}))$ is a simple module  by Theorem \ref{thm-irr}, and $V_{\mathcal{D}_m}(T(v+\delta^{m-1,1}))\simeq N/M$. Hence, $M$ is a maximal submodule of $N$. Finally, given that $E_{m1}$ acts bijectively on $N$ we get that $N\simeq D_{F_m }M$ by Proposition \ref{loc-max}. This completes the proof in the case $k=1$. 

Now, suppose that statement holds for all subsets of $F$ with $k-1$ elements. Define $F_i:=F\setminus {\{E_{m_i1}\}}$ and $\mathcal{D}_i:= \mathcal{D} \cup \{((m_i,1);(m_i-1,1)) \}$ for any $i=1,\dots, k$. Then $D_{F_i}M\simeq V_{\mathcal{D}_i}(T(v))$. Set $L=V_{\mathcal{D}_1}(T(v))+\dots +V_{\mathcal{D}_k}(T(v))$. As $D_{F_i}M \subset D_{\{E_{m_i1}\}}D_{F_i}M\simeq D_F M$, we have $L \subset D_F M$. Since $L$ is not $F$-bijective, then $L$ is a proper submodule of $D_F M$.
On the other hand, let $T(w)=T(v+i_1 \delta^{m_1-1,1}+(i_1-1) \delta^{m_2-1,1}+\dots+ \delta^{m_{i_1}-1,1} + \dots +(k-i_s+1) \delta^{m_{i_s+1}-1,1}+\dots +\delta^{m_k-1,1})$. Consider $\mathcal{A}= \mathcal{D} \cup \{((m_i-1,1);(m_i,1))\  |\ i=1,\dots,k \}$, where $\{m_i\ |\ i=1,2,\dots k \}=\{m_1,\dots, m_{i_1} \} \cup \dots \cup \{m_{i_s},\dots, m_k \} $ is a disjoint union of sets with consecutive elements. Then $\mathcal{A}$ is a maximal set of relations satisfied by $T(w)$ and $V_{\mathcal{A}}(T(w))$ is a simple module by Theorem \ref{thm-irr}. As $N/L\simeq V_{\mathcal{A}}(T(w))$, we conclude that $L$ is a maximal submodule of $N$ and $D_F M \simeq N$ by Proposition \ref{loc-max}.

\end{proof}

\section{Simple  modules in the minimal nilpotent orbit}\label{Section: Minimal orbit}
 Recall that  $\mathfrak{g}=\mathfrak{sl}_{n+1}$.

\subsection{Minimal nilpotent orbit}
Let $k$ an admissible number for $\widehat{\mathfrak{g}}$ with denominator $q\in \mathbb{N}$. 
 In this section we discuss explicit construction of simple admissible highest weight and $\mathfrak{sl}_2$-induced $\mathfrak{g}$-modules  in the minimal nilpotent orbit $\mathbb{O}_{min}$. 
The orbit $\mathbb{O}_{min}$ is the unique minimal non-trivial nilpotent orbit of $\mathfrak{g}$ with $\dim \mathbb{O}_{min}= 2n$.
We have the following description of  $[{\overline{Pr}_k^{\mathbb{O}_{min}}}]$:

\begin{Pro} \cite[Proposition 2.10]{AFR17}\label{Pro:minimal}
Then
 \begin{align*}
  [{\overline{Pr}_k^{\mathbb{O}_{min}}}]
  =\bigsqcup_{a=1}^{q-1}\{[\bar \lambda-\frac{ap}{q}\varpi_1]\mid \lambda\in \widehat{P}_+^{p-n-1}\},
 \end{align*}
 where $\widehat{P}_+^{p-n-1}$
is the set of  level $p-n-1$  integral dominant weights of $\widehat{\mathfrak {g}}$ and $\varpi_1$ is the first fundamental weight.
 \end{Pro}
 
\
To describe simple admissible $\mathfrak{g}$-modules in the minimal orbit $\mathbb{O}_{min}$  of level $k$ we need to find those  simple  $\mathfrak{g}$-modules $V$ for which  $\text{Ann}_{U(\mathfrak{g})}V=J_{\lambda}$, for each $\lambda\in [\overline{Pr}_k^{\mathbb{O}_{min}}]$. We start with the highest weight modules.

\

\subsection{Explicit realization of highest weight modules}  
Let
$$
k+n=\frac{p}{q}-1,\quad  p> n,\ q\geq 1 \mbox{ and }\ (p,q)=1.
$$
By Proposition \ref{Pro:minimal}, an element of  $[\overline{Pr}_k^{\mathbb{O}_{min}}]$  has the form 
$$
\bar\lambda-\frac{ap}{q}\varpi_1=\left(\lambda_{1}-\frac{ap}{q},\lambda_{2},\lambda_{3},\ldots,\lambda _{n-1},\lambda_n\right),
$$
where $\lambda_{i}\in\mathbb{Z}_{\geq 0}$, for all $ i=1,\dots ,n$  are such that $\lambda_{1}+\ldots+\lambda_n< p-n$ and $a\in \{1,2,\dots ,q-1\}$.

\begin{Ex}\label{ExMinOrbit} Let 
\textbf{$\mathfrak{g}=\mathfrak{sl}_3$}. As $\Lambda_1=(\lambda_1-\frac{ap}{q},\lambda_2) $  is regular dominant, by Corollary \ref{dom-ann},
 the simple admissible highest weight modules  in the minimal  orbit  are $L(\Lambda_i)$, $i=1,2,3$, where:
\begin{itemize}
    \item $\Lambda_2=s_1\cdot \Lambda_1=(\frac{ap}{q}-\lambda_1-2,\lambda_1+\lambda_2-\frac{ap}{q}+1)$;
\item  $\Lambda_3=s_2s_1\cdot \Lambda_1=(\lambda_2,\frac{ap}{q}-\lambda_1-\lambda_2-3)$.
\end{itemize}
These  modules have weight multiplicities bounded by $\lambda_2+1$.
\end{Ex}

Applying Corollary \ref{prop:relationbounded}, b), we get

 \begin{Th}\label{Pro: HWM} 
  Any simple admissible highest weight  module in the minimal nilpotent orbit is a bounded $\Gamma$-relation Gelfand--Tsetlin module.
 \end{Th}

\begin{Ex} Let  \textbf{$\mathfrak{g}=\mathfrak{sl}_4$} and  $\Lambda_1=(\lambda_1-\frac{ap}{q},\lambda_2,\lambda_3)$. 
Then  the simple admissible highest weight modules  in the minimal  orbit  are $L(\Lambda_i)$, $i=1,2,3,4$, where:

\begin{itemize}
    \item $\Lambda_2=s_1\cdot \Lambda_1=(\frac{ap}{q}-\lambda_1-2,\lambda_1-\frac{ap}{q}+\lambda_2+1,\lambda_3)$;
\item $\Lambda_3=s_2s_1\cdot \Lambda_1=(\lambda_2,\frac{ap}{q}-\lambda_1-\lambda_2-3,\lambda_1+\lambda_2+\lambda_3-\frac{ap}{q}+2)$;
\item $\Lambda_4=s_3s_2s_1\cdot \Lambda_1=(\lambda_2,\lambda_3, \frac{ap}{q}- \lambda_1-\lambda_2-\lambda_3-4)$.
\end{itemize}
These modules are bounded, e.g. the weight multiplicities of 
$L(\Lambda_1)$ are bounded by $\frac{1}{2}(\lambda_2+1)(\lambda_3+1) (\lambda_2+\lambda_3+2)$.

\end{Ex}

Let $F:=\{E_{m_i1}\ |\ i=1,\dots, k \}$ such that $m_i \in \{2,\dots, n+1\}$ for each $i=1,\dots, k$. From Corollary \ref{Localization annihilator} and Proposition \ref{localiz-F} we immediately obtain 
 
\begin{Co}\label{cor-singular}
Let $n\geq 2$.
All simple subquotients of   $D_F L(\lambda')$ are admissible bounded  $\Gamma$-relation Gelfand--Tsetlin $\mathfrak g$-modules  in the minimal  orbit.
\end{Co}

\
\begin{Rem}
\begin{enumerate}[(i)] 
\item  Corollary \ref{cor-singular} for $\mathfrak{sl}(3)$ was shown in   \cite[Theorem 5.6]{AFR17}.
 \item All simple modules  in Corollary \ref{cor-singular}  are highest weight modules (with respect to some Borel subalgebra) with bounded weight multiplicities.
 \end{enumerate}
\end{Rem}
\
We have  from Corollary \ref{cor-outras Borel}

\begin{Co}\label{cor-diff-Borel-same}
 Let ${\mathfrak b}={\mathfrak b}(\pi)$, $\beta_j\in \pi$, $j=1, \ldots, t $ with
$\left<\lambda+\rho_{\pi}, \beta_j^{\vee} \right> \in\mathbb{Z}_{\geq 0}$ for all $j$, and let  $w\in W$ be such that $\left<w(\lambda+\rho_{\pi}), \beta_j^{\vee} \right> \in\mathbb{Z}_{< 0}$ for all $j=1, \ldots, t$.   Then 
  $L(\lambda)\simeq L_{\mathfrak b(w\pi)}(w\lambda)$  is a $\Gamma_{\mathcal F}$-relation Gelfand--Tsetlin module, where $\Gamma_{\mathcal F}=w\Gamma$.
\end{Co}

\

We also have the following result.

\begin{Co}\label{lem-diff-Borel} Let $\lambda\in\mathfrak{h}^*$, $\pi$ a basis of the root system, 
 $L(\lambda)=L_{\mathfrak b(\pi)}(\lambda)$ an admissible highest weight module in the minimal  orbit (with respect to the Borel subalgebra $\mathfrak b(\pi)$) and    $\beta\in \pi$ is  such that $\left<\lambda,\beta^{\vee} \right> \notin\mathbb{Z}_{\geq 0}$. Then
 
\begin{enumerate}[a)] 
\item The module $L_{\mathfrak b(s_{\beta}\pi)}(s_{\beta}(\lambda+\beta))$ 
 is an admissible $s_{\beta}\Gamma$-relation Gelfand--Tsetlin module in the minimal  orbit.
 \item Let $i>1$, $\beta=\beta_i$ the first simple root  of $\pi$ such that $\left<\lambda,\beta^{\vee} \right> \notin\mathbb{Z}_{\geq 0}$
  and $w=s_{i-1} s_i s_{i-2} s_{i-1} \cdots  
s_2s_1s_3s_1s_2 \in W$. Then $L_{\mathfrak b(w^{-1}\pi)}(w^{-1}\lambda)$ is a $\Gamma$-relation Gelfand--Tsetlin module in the minimal  orbit.
 \item For any $x\in \mathbb C$, $D_{f_{\beta}}^x L(\lambda)$ is a $w\Gamma$-relation Gelfand--Tsetlin module in the minimal  orbit.
\end{enumerate}
\end{Co}

\begin{proof}
 Let $\Lambda=$
$
\bar\lambda-\frac{ap}{q}\varpi_1$, where  $\bar\lambda=(\lambda_1, \ldots, \lambda_n)
$ with non-negative integers $\lambda_i$ for all $i=1, \ldots, n$. Then $\lambda=s_t\ldots s_1\cdot \Lambda$ for some $t\leq n$. If $t=n$ then $\lambda$ has the last component non-integral in which case $\beta=\alpha_n$. If $t<n$ then $\lambda$ has exactly two non-integral components in the places $t$ and $t+1$ and $\beta=\alpha_t$ or $\beta=\alpha_{t-1}$. Hence, 
$L(s_{\beta}\cdot \lambda)$ is a simple admissible highest weight module  in the minimal  orbit. Note that
$s_{\beta}\cdot \lambda=s_{\beta}(\lambda+\beta)$ and consider $L_{\mathfrak b(s_{\beta}\pi)}(s_{\beta}(\lambda+\beta))$. This is the  highest weight module with respect to the Borel subalgebra $\mathfrak b(s_{\beta}\pi)$ and the corresponding highest weight (with respect to $\mathfrak b(s_{\beta}\pi)$) is $s_{\beta}(\lambda+\beta)$. Therefore, $L_{\mathfrak b(s_{\beta}\pi)}(s_{\beta}(\lambda+\beta))$ 
 is an admissible module in the minimal  orbit, and it is a $\Gamma'$-relation Gelfand-Tsetlin module 
 where $\Gamma'$ is the standard Gelfand-Tsetlin subalgebra of $s_{\beta}\pi$, that is 
 $\Gamma'=s_{\beta}\Gamma$. This shows a). 
 
  Since $L(\lambda)$ is a module in the nilpotent orbit, then $L_{w^{-1}\mathfrak b}(w^{-1}\lambda)$ is also an admissible module in the nilpotent orbit. Hence,  $L_{\mathfrak b(w^{-1}\pi)}(w^{-1}\lambda)$ is a $w^{-1}\Gamma$-relation Gelfand--Tsetlin module. The statement b) follows from  Lemma \ref{lem-key}.
 Now, twisting $L_{\mathfrak b(w^{-1}\pi)}(w^{-1}\lambda)$ by $w$, that is applying $w$ to the corresponding Gelfand-Tsetlin formulas, we obtain that $L(\lambda)$ is a $w\Gamma$-relation Gelfand--Tsetlin module.
 Note that $w\in W$ is the element of minimal length such that $\beta$ is the first simple root of $w\pi$. Hence,  c) follows from Theorem \ref{twistedE21}.
\end{proof}

\

\subsection{Classification of \texorpdfstring{$\mathfrak{sl}_2$}{sl2}-induced modules}
Let  $k=\frac{p}{q}-n-1$ be an admissible number for $\widehat{\mathfrak{g}}$.
 In  this section we complete the construction of all simple admissible $\mathfrak{g}$-modules in the minimal orbit, which are the quotients of modules induced  from  parabolic subalgebras with the Levi factor isomorphic to  $\mathfrak{sl}_2+\mathfrak{h}$. All such modules are $\Gamma_{\mathcal F}$-relation modules for some $\mathcal F$.

Let $M$ be an admissible  $\mathfrak{g}$-module of level $k$. Consider a parabolic subalgebra  $\mathfrak{p}=\mathfrak{p}_{\Sigma}=\mathfrak{l}_{\Sigma}+\mathfrak{n}_{\Sigma}$  of $\mathfrak{g}$, where
 $\Sigma$ consists of one simple root $\beta$. Denote by  $M^{\mathfrak{n}}$ the subspace of all $\mathfrak{n}_{\Sigma}$-invariants, which is  an  $\mathfrak{l}_{\Sigma}$-module.
Suppose  $M=L_{\mathfrak p}^{\mathfrak g}(\Sigma, N)$ for some simple weight $\mathfrak{l}_{\Sigma}$-module $N$ 
with $\mathfrak{n}_{\Sigma}N=0$. 
Then   
$M^{\mathfrak{n}}\simeq N$  is  admissible $\mathfrak{l}_{\Sigma}$-module of level $k_{\beta}=\frac{2}{(\beta,\beta)}(k+n+1)-2$ by  \cite[Theorem 2.12]{AFR17}. 
\

The following is straightforward.

\begin{Pro}\label{prop-sl2}
Let $\mathfrak a=\mathfrak{sl}_2$, $V=V(\gamma, \mu)$  a simple dense weight $\mathfrak a$-module, $\gamma, \mu\in \mathbb C$.  Then $V$ is admissible of level $k$ in the minimal orbit if and only if   $\mu=\lambda-  \dfrac{ap}{q}+2x$ and  $\gamma=\left(\lambda-  \dfrac{ap}{q}+1 \right)^2$, where $\lambda\in \{0,1,\dots, p-2 \}$, $a\in \{ 1, \ldots, q-1 \}$, $x\in  \mathbb{C}\setminus{\mathbb{Z}}$ and $x-\dfrac{ap}{q} \notin \mathbb{Z}$.
\end{Pro}

\begin{Rem} In the proposition above 
 we have  $V=V_{\mathcal{\mathcal{C}}}(T(v))$, where $\mathcal{C}=\emptyset$, $T(v)=T(v_{ij})$ is the Gelfand--Tsetlin tableau with height 2,  such that $v_{11}=\dfrac{\mu}{2}$,  $v_{21}=\dfrac{1}{2}\left(\lambda-\dfrac{ap}{q} \right)$, and $v_{22}=\dfrac{1}{2} \left(\dfrac{ap}{q}-\lambda- 2\right)$ (cf. Theorem \ref{sl2 Induced}).  The set $\mathcal{B}_{\mathcal{\mathcal{C}}}(T(v))= \{ T(v+\ell \delta^{11}) \ |\  \ell \in \mathbb{Z} \}$ is a basis of $V$.
\end{Rem}

\

Applying Proposition \ref{prop-sl2}, Theorem \ref{Th;subsquence} and Theorem \ref{sl2 Induced}
we obtain the following statement.

\begin{Co}\label{minimalinducedsl2}  Let $n>1$, $\gamma\in \mathbb C$, $\mu=(\mu_1, \ldots, \mu_n)\in \mathfrak{h}^*$, such that 
$V\simeq V(\gamma, \mu_1)$ is a simple dense weight $\mathfrak{sl}_2$-module and $\mu_i=\left<\mu, \alpha_i^{\vee} \right>$, 
$i=1, \ldots, n$.
Let $L(\gamma,\mu)=L_\mathfrak{p}^{\mathfrak{g}}(\{ \alpha \}, V)$.  Then $L(\gamma,\mu)$ is admissible of level $k$ in the minimal orbit if and only if   $\mu_1=\lambda_1-  \dfrac{ap}{q}+2x$, $\mu_2=\lambda_2-x$, $\mu_j=\lambda_j$ for all $j=3,\dots, n$ and  $\gamma=\left(\lambda_1-  \dfrac{ap}{q}+1 \right)^2$ with $\{\lambda_{i}\}_{i=1,\ldots,n} \subset \mathbb{Z}_{\geq 0}$ such that $\lambda_{1}+\ldots+\lambda_n< p-n$,  $a\in \{1,2,\dots ,q-1\}$, $x\in  \mathbb{C}\setminus{\mathbb{Z}}$  and $x-\dfrac{ap}{q} \notin \mathbb{Z}$.
\end{Co}


\

\begin{Th}\label{inducedsl2minimal} Let $\pi$ be a basis of the root system of $\mathfrak g$,
$\beta$  a positive root of $\mathfrak g$ (with respect to $\pi$).  Let $L_{\mathfrak b(\pi)}(\lambda)$ be an admissible simple $\mathfrak b(\pi)$-highest weight $\mathfrak g$-module  in the minimal  orbit,  such that  $\left<\lambda,\beta^{\vee} \right> \notin\mathbb{Z}$, and $f=f_{\beta}$. Denote by $A_{\pi, \beta}$ the set of all $x\in \mathbb{C}\setminus \mathbb{Z}$ such that   $x+\left<\lambda+\rho_{\pi},\beta^{\vee} \right> \notin\mathbb{Z}$.

 \begin{enumerate}[a)]
 \item The $\mathfrak g$-module $D^x_{f}L_{\mathfrak b(\pi)}(\lambda)$ is admissible    in the minimal  orbit for any $x\in A_{\pi, \beta}$;
  \item
    Modules $D^x_{f}L_{\mathfrak b(\pi)}(\lambda)$, where $\pi$ runs the sets of simple roots of $\mathfrak g$, $\beta$ runs positive roots with respect to $\pi$,   $x\in A_{\pi, \beta}$, $\left<\lambda,\beta^{\vee} \right> \notin\mathbb{Z}$ and $L_{\mathfrak b(\pi)}(\lambda)$ is admissible module in the minimal  orbit, exhaust all  simple $\mathfrak{sl}_2$-induced admissible modules in the minimal orbit.  All such modules have bounded weight multiplicities;
     \item  There exists a flag ${\mathcal F}$ such that $D^x_{f}L_{\mathfrak b(\pi)}(\lambda)$ is  $\Gamma_{\mathcal F}$-relation Gelfand--Tsetlin $\mathfrak g$-module. 
 \end{enumerate}
\end{Th}

\begin{proof} Let $\pi=\{\alpha_1, \ldots, \alpha_n\}$.
First we prove a).
 Let $\beta=\alpha_r+\ldots +\alpha_t$ for some consecutive simple roots
$\alpha_j$, $j=r, \ldots, t$. Since $\lambda$ is admissible in the minimal orbit then we have two possibilities: either there exists only one simple root $i$ such that  $\left<\lambda,\alpha_i^{\vee} \right> \notin\mathbb{Z}$, or there are only two such roots which are consecutive. Assume the first case. Then either $i=r=1$ or $i=t=n$. Without loss of generality we may assume that $i=1$ (if $i=n$ then apply symmetry of the root system), and hence  $\left<\lambda+\rho_{\pi},\alpha_j^{\vee} \right> \in\mathbb{Z}_{> 0}$, $j=2, \ldots, t$. 
Take the following basis of the root system
$$\pi'=\{\beta,  -\alpha_t, \ldots, -\alpha_{2},  \alpha_{2}+\ldots +\alpha_{t+1}, \alpha_{t+2},
\ldots, \alpha_{n}\}.$$
Let $w\in W$ be such that $w\pi=\pi'$. Then $w$ satisfies Corollary  \ref{cor-diff-Borel-same}, and 
  $L(\lambda)\simeq L_{\mathfrak b(w\pi)}(w\lambda)$  is a $w\Gamma$-relation Gelfand--Tsetlin module. Then the statement follows from   Theorem \ref{twistedE21}.
  
  Consider now the second case. 
   Assume that $\left<\lambda,\alpha_i^{\vee} \right> \notin\mathbb{Z}$ for $i=k,k+1$ for some $k$. If $\beta=\alpha_k$ then the statement follows from Corollary \ref{lem-diff-Borel}, c). If $\beta=\alpha_{k+1}$ then apply the symmetry of the Dynkin diagram and  Corollary \ref{lem-diff-Borel}, c).
    Let $\beta=\alpha_{r}+\ldots +\alpha_k$ for some $1\leq r\leq k-1$. Take the minimal $w\in W$ such that $w\pi$ contains $\beta$ and 
  $-\alpha_{r}, \ldots, -\alpha_{k-1}$ (such $w$ clearly exists). Then  $L(\lambda)\simeq L_{\mathfrak b(w\pi)}(w\lambda)$. Hence, the problem reduces to the case $\beta=\alpha_k$ which was argued above. If $\beta=\alpha_{k+1}+\ldots +\alpha_t$ for some $k+1\leq t\leq n$ then the statement follows from the symmetry of the Dynkin diagram. Now b) and c) follow from a) and Theorem \ref{sl2 Induced}.


\end{proof}

We note that Theorem \ref{inducedsl2minimal} was initially proved for $\mathfrak g=\mathfrak{sl}(3)$ in \cite[Theorem 5.6]{AFR17}.

Let $\beta$ be a root of $\mathfrak g$ and $\pi$ be a basis of the root system containing $\beta$ as the first root (such $\pi$ always exists by the  conjugation by the Weyl group). 
Let ${\mathfrak p}=\mathfrak{a}_{\beta}\oplus \mathfrak{n}$ be a parabolic subalgebra of $\mathfrak g$ containing $\mathfrak b(\pi)$ with the Levi factor $\mathfrak{a}_{\beta}\simeq\mathfrak{sl}_2+\mathfrak h$ based on the root $\beta$. Let $V=V(\gamma_{\beta}, \mu)$, $\gamma_{\beta}\in \mathbb C$,  $\mu=\mu_1 \varpi_1+\ldots + \mu_n \varpi_n \in \mathfrak h^*$, where $\gamma_{\beta}$ is the eigenvalue of the Casimir element of $\mathfrak{a}_{\beta}$ and $\left<\varpi_i, \beta_j^{\vee} \right> = \delta_{ij}$, $i=1, \ldots, n$. 
  
\

  \begin{Co} The module
 $L_{\pi, \beta}(\gamma_{\beta}, \mu):=L_{\mathfrak p}^{\mathfrak g}(V)$    
 is  admissible   in the minimal  orbit 
   if and only if   $\left<\mu-\lambda,\beta_1^{\vee} \right>=2x$, $\left<\mu-\lambda,\beta_2^{\vee} \right>=-x$, $\left<\mu-\lambda,\beta_i^{\vee} \right>=0$ for each $i=3,\dots, n$, and $\gamma_{\beta}= \left<\lambda+\rho_{\mathfrak{b}},\beta_1^{\vee} \right>^2$, with  $\left<\lambda+\rho_{\pi},\beta_i^{\vee} \right> \in \mathbb{Z}_{> 0}$  for all $i\in \{2,\ldots,n\} $,   $x\in  \mathbb{C}\setminus{\mathbb{Z}}$,  $x+\left<\lambda+\rho_{\pi},\beta_1^{\vee} \right> \notin \mathbb{Z}$ and $\left<\lambda+\frac{ap}{q}\varpi_1+\rho_{\pi},\beta_1^{\vee} \right> \in \mathbb{Z}_{> 0}$, $\left<\lambda+\frac{ap}{q}\varpi_1+ \rho_{\pi},\beta_{1,n}^{\vee} \right> < p$ for some  $a\in \{1,2,\dots ,q-1\}$.
  
\end{Co}

\

\begin{Rem}
Theorem \ref{inducedsl2minimal} provides an algorithm how to list all simple $\mathfrak{sl}_2$-induced admissible modules in the minimal orbit:
\begin{itemize}
\item Consider all possible Borel subalgebras of $\mathfrak g$ containing $\mathfrak h$;
\item For each Borel subalgebra $\mathfrak b(\pi)$ describe $\lambda\in \mathfrak h^*$ for which 
$L_{\mathfrak b}(\lambda)$ is admissible  using the Arakawa's classification for the standard Borel and applying the Weyl group;
\item Choose any positive  (with respect to $\pi$) root $\beta$ such that $\left<\lambda,\beta^{\vee} \right> \notin\mathbb{Z}$, and define $ D^x_{f}L_{\mathfrak b(\pi)}(\lambda)$ for any $x\in \mathbb{C}\setminus \mathbb{Z}$ such that  $x+\left<\lambda+\rho_{\pi},\beta^{\vee} \right> \notin\mathbb{Z}$, where $f=f_{\beta}$.
\end{itemize} 
Obtained modules exhaust all simple $\mathfrak{sl}_2$-induced admissible modules in the minimal orbit.
Moreover, the proof of Theorem \ref{inducedsl2minimal} explains how to define  the flag  ${\mathcal F}$ for which $L_{\pi, \beta}(\gamma_{\beta}, \mu)$ is a $\Gamma_{\mathcal F}$-relation Gelfand--Tsetlin $\mathfrak g$-module, and hence to obtain explicit tableaux realization for all such modules. 
\end{Rem}

\end{document}